\documentclass[reqno,centertags, 12pt]{amsart}
\usepackage{amsmath,amsthm,amscd,amssymb}
\usepackage{esint}  
\usepackage{latexsym}
\usepackage{graphicx}
\usepackage{enumitem}
\usepackage[mathscr]{eucal}

\usepackage{hyperref}

\newcommand{\bbC}{{\mathbb{C}}}
\newcommand{\bbD}{{\mathbb{D}}}

\newcommand{\bbF}{{\mathbb{F}}}
\newcommand{\bbG}{{\mathbb{G}}}
\newcommand{\bbH}{{\mathbb{H}}}

\newcommand{\bbR}{{\mathbb{R}}}

\newcommand{\bbZ}{{\mathbb{Z}}}

\newcommand{\calA}{{\mathcal{A}}}
\newcommand{\calB}{{\mathcal{B}}}
\newcommand{\calC}{{\mathcal{C}}}
\newcommand{\calE}{{\mathcal{E}}}

\newcommand{\calG}{{\mathcal G}}
\newcommand{\calH}{{\mathcal H}}
\newcommand{\calI}{{\mathcal I}}

\newcommand{\calL}{{\mathcal L}}

\newcommand{\calP}{{\mathcal P}}

\newcommand{\calT}{{\mathcal T}}

\newcommand{\calV}{{\mathcal V}}
\newcommand{\calW}{{\mathcal W}}
\newcommand{\bdone}{{\boldsymbol{1}}}

\newcommand{\lb}{\label}

\newcommand{\wti}{\widetilde  }

\newcommand{\tr}{\text{\rm{Tr}}}

\newcommand{\spec}{\text{\rm{spec}}}

\newcommand{\supp}{\text{\rm{supp}}}

\newcommand{\bi}{\bibitem}

\newcommand{\beq}{\begin{equation}}
\newcommand{\eeq}{\end{equation}}
\newcommand{\ba}{\begin{align}}
\newcommand{\ea}{\end{align}}

\renewcommand{\Im}{\operatorname{Im}}

\let\det=\undefined\DeclareMathOperator{\det}{det}

\newcounter{smalllist}
\newenvironment{SL}{\begin{list}{{\rm\roman{smalllist})}}{%
\setlength{\topsep}{0mm}\setlength{\parsep}{0mm}\setlength{\itemsep}{0mm}%
\setlength{\labelwidth}{2em}\setlength{\leftmargin}{2em}\usecounter{smalllist}%
}}{\end{list}}

\newcommand{\bigtimes}{\mathop{\mathchoice%
{\smash{\vcenter{\hbox{\LARGE$\times$}}}\vphantom{\prod}}%
{\smash{\vcenter{\hbox{\Large$\times$}}}\vphantom{\prod}}%
{\times}%
{\times}%
}\displaylimits}

\newcommand{\comm}[1]{}

\let\hom\relax
\DeclareMathOperator{\hom}{Hom}

\allowdisplaybreaks
\numberwithin{equation}{section}

\newtheorem{theorem}{Theorem}[section]
\newtheorem{proposition}[theorem]{Proposition}
\newtheorem{lemma}[theorem]{Lemma}
\newtheorem{corollary}[theorem]{Corollary}
\theoremstyle{definition}
\newtheorem{example}[theorem]{Example}
\newtheorem{conjecture}[theorem]{Conjecture}
\newtheorem{fact}[theorem]{Fact}
\newtheorem{problem}[theorem]{Problem}
\newtheorem{guess}[theorem]{Initial Guess}

\newtheorem*{remark}{Remark}
\newtheorem*{remarks}{Remarks}

\newcommand{\jap}[1]{\langle #1 \rangle}
\newcommand{\norm}[1]{\lVert#1\rVert}

\begin{document}

\title[Periodic Trees]{Periodic Jacobi Matrices on Trees}
\author[N.~Avni, J.~Breuer and B.~Simon]{Nir Avni$^{1,4}$, Jonathan Breuer$^{2,5}$ \\and Barry Simon$^{3,6}$}

\thanks{$^1$ Department of Mathematics, Northwestern University, Evanston, IL  E-mail: avni.nir@gmail.com}

\thanks{$^2$ Institute of Mathematics, The Hebrew University of Jerusalem, Jerusalem, 91904, Israel. E-mail: jbreuer@math.huji.ac.il.}

\thanks{$^3$ Departments of Mathematics and Physics, Mathematics 253-37, California Institute of Technology, Pasadena, CA 91125, USA. E-mail: bsimon@caltech.edu.}

\thanks{$^4$ Research supported in part by NSF grant DMS-1902041.}

\thanks{$^5$ Research supported in part by Israeli BSF Grant No. 2014337. and Israel Science Foundation Grant No. 399/16}

\thanks{$^6$ Research supported in part by NSF grant DMS-1665526 and in part by Israeli BSF Grant No. 2014337.}

\

\date{\today}

\begin{abstract}  We begin the systematic study of the spectral theory of periodic Jacobi matrices on trees including a formal definition.  The most significant result that appears here for the first time is that these operators have no singular continuous spectrum.  We review important previous results of Sunada and Aomoto and present several illuminating examples.  We present many open problems and conjectures that we hope will stimulate further work.
\end{abstract}

\maketitle

\section{Introduction} \lb{s1}

One of the crowning achievements of spectral theory and of mathematical physics is the theory of periodic Schr\"{o}dinger operators especially in one dimension and the related theory of periodic Jacobi matrices.  In this paper, we discuss a variety of aspects of a class of objects that we feel are an interesting analog of these one dimensional objects, in many ways closer to the one dimensional periodic theory than the theory on $\bbZ^\nu$ or $\bbR^\nu$ when $\nu>1$.  Remarkably, there is almost no study of these objects in the mathematical physics or spectral theory literature and we hope this paper will stimulate those communities to their further study.

These objects are periodic Jacobi matrices on discrete trees.  In  general, a Jacobi matrix on a graph is an operator that acts on $\ell^2$ of the vertices of the graph with non-zero matrix elements only for pairs of indices that are the two vertices at the ends of an edge of the graph or diagonal matrix elements.  It is not even clear what it means for such a Jacobi matrix to be periodic and one of our goals here is to give a precise definition: If $\calG$ is a finite graph without leaves, then its universal cover, $\calT$, is a tree on which the fundamental group of $\calG$ acts. If $J$ is a Jacobi matrix on $\calG$, it has a natural lift to $\calT$ which commutes with the set of unitaries on $\ell^2(\calT)$ induced by the fundamental group.

What makes these new objects fascinating is that the underlying group is not Abelian (unless $\calT=\bbZ$) and what makes the theory difficult is that there is no longer a Floquet analysis based on the one dimensional irreducible representation of the underlying symmetry group.

There has been some study of the objects we study (albeit without a formal definition of what they are!) among the community of researchers that has focused on the theory of Laplacians and adjacency matrices on infinite graphs, a community which unfortunately has not had close connection to others working in spectral theory. We believe there are, so far, three major results on these objects:

(1) \emph{Gap labeling}.  If $p$ is the underlying period, then the density of states in a gap of the spectrum is $j/p\, (j\in\{1,2,\dots,p-1\})$.  In particular there are at most $p-1$ gaps in the spectrum and so a band structure.  This result is implicit in Sunada \cite{Sun} who proves a band structure for certain continuum Schr\"{o}dinger operators on manifolds with a hyperbolic symmetry and remarks in passing that the ideas also hold for suitable discrete operators.

(2) \emph{No flat bands on regular trees}. In Aomoto \cite{AomotoPoint}, it is proven that if a tree is regular (i.e.\ of constant degree), then periodic Jacobi matrices have no point spectrum.  On the other hand, the same paper describes an example of a non-regular periodic tree with a Jacobi matrix that has point spectrum.

(3) \emph{Absence of singular continuous spectrum}. The result, new here, that periodic Jacobi matrices on trees have no singular continuous spectrum.

One of our main goals in this paper is to discuss these three results.  Another is to describe a few explicit models which we feel are illuminating.  Finally, we intend to discuss a number of conjectures and open questions that we feel could be stimulating.

Section \ref{s2} summarizes the 1D case which serves as the source of some of our conjectures. Section \ref{s3} makes precise our definition of periodic Jacobi matrices and presents a definition of the period.  Section \ref{s4} defines the density of states and discusses its relation to eigenvalue counting densities while Section \ref{s5} presents gap labeling.  The appendix
 provides an exposition of the basic gap labeling result which we feel will be more accessible to spectral theorists. Section \ref{s6} discusses linked equations for the Green's and m--functions which we then use to prove the absence of singular continuous spectrum.  Section \ref{s7} has some of the promised illuminating examples and Section \ref{s8} discusses the context for Aomoto's result on no point spectrum in the case of regular trees.  Sections \ref{s9} and \ref{s10} presents a number of open questions and conjectures.

We'd like to thank a number of people for illuminating discussions: Nalini Anantharaman, Jacob Christiansen, Latif Eliaz, Alexandre Eremenko, Matthias Keller, Wolfgang Woess and Maxim Zinchenko.

\section{Highlights in One Dimension} \lb{s2}

As we explained, thinking of 1D Jacobi matrices as regular degree 2 trees makes classical two-sided Jacobi matrices our guide for what to look for in the analysis of general trees, so to set notation and expectations, we briefly summarize the beautiful theory of 1D periodic Jacobi matrices.  A reference for much of this is \cite[Chapters 5 and 6]{SiSz}.

Our operators, $J$, act on $\ell^2(\bbZ)$ and depend on two sequences of real numbers $\{a_n\}_{n\in\bbZ}$ and $\{b_n\}_{n\in\bbZ}$ with $a_n>0$.  The Jacobi matrix indexed by those sequences acts by
\begin{equation}\label{2.1}
  (Ju)_n = a_n u_{n+1} + b_n u_n + a_{n-1} u_{n-1}
\end{equation}
It is natural to think of $\bbZ$ as a graph with vertices $j\in\bbZ$ and edges from $j$ to $j+1$.  $b_j$ is associated to the vertex at $j$ and $a_j$ to the edge from $j$ to $j+1$.

We suppose that J is periodic, that is for some $p\in\bbZ_+$, we have that
\begin{equation}\label{2.2}
  a_{k+p}=a_k, \qquad b_{k+p}=b_k
\end{equation}
for all $k\in\bbZ$.

For a bounded measurable function, $f$, one defines $f(J)$, an operator on $\ell^2(\bbZ)$ by the spectral theorem \cite[Section 5.1]{OT}.  The spectral measures, $d\mu_n, n\in\bbZ$, are defined by
\begin{equation}\label{2.3}
  \jap{\delta_n,f(J)\delta_n} = \int f(\lambda)\,d\mu_n(\lambda)
\end{equation}
For $\Omega\subset\bbR$, a Borel set and $\chi_\Omega$ the characteristic function of $\Omega$, $\chi_\Omega(J)$ is the spectral projection, $P_\Omega(J)$.  Because of \eqref{2.2}, we have that
\begin{equation}\label{2.3a}
  d\mu_{k+p} = d\mu_k
\end{equation}

\textbf{A. The DOS and Gap Labeling} For $a,b\in\bbZ$, we let $\chi_{[a,b]}$ be the operator on $\ell^2(\bbZ)$ that is the orthogonal projection onto vectors supported on $[a,b]$.  By \eqref{2.2} and \eqref{2.3a}, we have that

\begin{fact} \lb{F2.1}
\begin{equation}\label{2.4}
  \lim_{a\to -\infty;\,b\to\infty} \frac{1}{b-a}\tr(\chi_{[a,b]} f(J)) = \int f(\lambda)\,dk(\lambda)
\end{equation}
where
\begin{equation}\label{2.5}
  dk(\lambda) = \frac{1}{p}\sum_{j=1}^{p} d\mu_j(\lambda)
\end{equation}
\end{fact}

$k(\lambda) = \int_{\infty}^\lambda dk$ is called the \emph{integrated density of states} (IDS) and $dk$ the \emph{density of states} (DOS).  The DOS for periodic Schr\"{o}dinger operators goes back to the earliest days of condensed matter theory.  It seems to have entered the mathematical physics community in the more general consideration of ergodic discrete Schr\"{o}dinger operators; some of the results are due to Benderskii--Pastur \cite{BP}, Pastur \cite{P73}, Nakao \cite{N77}, Kirsch--Martinelli \cite{KM82} and Avron--Simon \cite{AS83} who also had Facts \ref{F2.2} and \ref{F2.3} below.

Since \eqref{2.5} is a finite sum, we have that $\supp(dk)=\cup_{j=1}^p \supp(d\mu_j)$ so

\begin{fact} \lb{F2.2} \spec(J)=\supp(dk)
\end{fact}

Let $J^{(k)}_{D,P}$ be $J$ restricted to $[1,kp]$ with Dirichlet (i.e.\ setting $a_0=a_{kp}=0$) or periodic (i.e.\ putting $a_0$ in the upper right and lower left corners) boundary conditions.  Then Avron--Simon \cite{AS83} proved that

\begin{fact} \lb{F2.3} $k(E)$ is an eigenvalue counting density, i.e. if $N^{(k)}_{D,P}(E)$ is the number of eigenvalues of $J^{(k)}_{D,P}$ below $E$, then
\begin{align}
  k(E) &= \lim_{k\to\infty} N^{(k)}_D(E)/kp \lb{2.6} \\
       &= \lim_{k\to\infty} N^{(k)}_P(E)/kp \lb{2.7}
\end{align}
\end{fact}

The following also goes back to the earliest days of quantum theory or even earlier given the work of Hill and Floquet.

\begin{fact} \lb{F2.4} (Band Structure) $\spec(J)$ is a union of at most $p$ disjoint closed intervals.
\end{fact}

\begin{fact} \lb{F2.5} (Gap Labeling) In each gap of $\spec(J)$ (i.e.\ connected bounded open interval in $\bbR\setminus\spec(J)$), we have that $k(E)=j/p$ for some $j \in \{1,\dots,p-1\}$.
\end{fact}

In the mathematical physics literature, the latter fact came to the fore with its extension to the almost periodic case \cite{AS83, JM, Bell}

\medskip
\textbf{B. Spectral Properties}

The basic result is that the spectral measures are purely absolutely continuous which we prefer to list as two facts

\begin{fact} \lb{F2.6} J has no singular continuous spectrum.
\end{fact}

\begin{fact} \lb{F2.7} J has no pure point spectrum.
\end{fact}

Because he proved the analog for higher dimensions, the last fact is often associated with L. Thomas \cite{Thomas}.

\medskip
\textbf{C. Analyticity of $m-$ and Green's Functions}

The Green's function is
\begin{equation}\label{2.8}
  G_n(z) = \jap{\delta_n,(H-z)^{-1}\delta_n}
\end{equation}
If we replace $a_{n-1}$ by $0$, then $H$ decomposes into a direct sum, $H_n^+$ acting on $\ell^2(n,\infty)$ and $H_{n-1}^-$ acting on $\ell^2(-\infty,n-1)$ and we define
\begin{equation} \lb{2.9}
  m_n^+(z) = \jap{\delta_n,(H_n^+-z)^{-1}\delta_n} \qquad m_n^-(z) = \jap{\delta_{n-1},(H_n^--z)^{-1}\delta_{n-1}}.
\end{equation}

The functions $G_n,m_n^{+}$, and $m_n^{-}$ are related by
\[
G_n(z)=\frac{1}{-z+b_n-a_n^2m_n^+(z)+a_{n-1}^2m_n^-(z)}.
\]

\begin{fact} \lb{F2.8} For all $n,\,G_n(z)$ and $m_n^\pm(z)$ have analytic continuations from  $\bbC\setminus$spec$(H)$ to a finite sheeted Riemann surface with a discrete set of branch points.
\end{fact}

\begin{fact}  \lb{F2.9}  These functions are hyperelliptic and, in particular, have only square root branch points and the surface is two sheeted.
\end{fact}

\begin{fact} \lb{F2.10} The branch points are all in $\bbR$ at edges of the spectrum.  There are no poles of $G$ away from the branch points and all poles of $m^\pm$ are in the bounded spectral gaps of one sheet or the other or at the branch points.  There is one pole in each ``gap''.
\end{fact}

These results follow by writing down an explicit quadratic equation for the $m$-functions (which follows from the recursive relation \eqref{6.5}) and analyzing it using, in part, the monotonicity of $G$ in gaps and the fact that poles of $m$ correspond to zeros of $G$.

\medskip
\textbf{D. Universality of the DOS}

Two periodic Jacobi matrices are called \emph{isospectral} if they have the same spectrum (as a set).

\begin{fact} \lb{F2.11}
  Two isospectral Jacobi matrices have the same period and same DOS.
\end{fact}

\begin{fact} \lb{F2.13}
  The DOS of a periodic Jacobi matrix is equal to the potential theoretic equilibrium measure, aka harmonic measure, of its spectrum.
\end{fact}

(Standard references for potential theory are \cite{ArmGard, Helms, Land, Rans, Werm} or \cite[Section 3.6]{HA}.) The second of these facts, together with the Borg-Hochstadt Theorem (Fact 2.14) below, implies the first.  For mathematical physicists, these facts are connected to the Thouless formula and the fact that pure a.c. spectrum implies the Lyaponov exponent is zero on the spectrum; see for example Simon \cite{SiKotani}.  In the OP community, it is connected to the theory of regular Jacobi matrices as developed especially by Stahl--Totik \cite{ST}.  We'll see in G below that these results plus gap labeling restrict the sets that can be spectra of periodic Jacobi matrices.

\medskip
\textbf{E. Borg and Borg--Hochstadt Theorems}

\begin{fact} \lb{F2.14} (Borg's Theorem) If a periodic Jacobi matrix has no gaps in its spectrum, then $a$ and $b$ are constant.
\end{fact}

\begin{fact} \lb{F2.15} (Borg--Hochstadt Theorem) If the IDS of a periodic Jacobi matrix has a value $j/p$ in each gap of the spectrum, then the period is (a divisor of) $p$.
\end{fact}

Both Borg \cite{Borg} and Hochstadt \cite{Hoch} proved their results for Hill's equation (i.e. continuum Schr\"{o}dinger operators) but it is known to hold for the Jacobi case; see for example \cite[Theorem 5.4.21 and Corollary 5.13.9]{SiSz}.

\medskip
\textbf{F. Floquet Theory and Spectral Gaps}

While it is often expressed in terms of Floquet boundary conditions, it is better for our purposes to consider the group of symmetries $W_n = U^n$ where $U$ is the symmetry $Uu_j=u_{j+p}$ so $W_nH=HW_n$.

\begin{fact} \lb{F2.16}
The representation of $\{W_n\}_{n\in\bbZ}$ acting on $\ell^2(\bbZ)$ is a direct integral of all the irreducible representations of $\bbZ$, each with multiplicity $p$
\end{fact}

\begin{fact} \lb{F2.17}
$H$ is a direct integral of the $p\times p$ matrices $H(\theta); \, e^{i\theta}\in\partial\bbD$ given by
\[
H(\theta)=\left( \begin{matrix} b_1 & a_1 & 0 & 0 & \cdots & 0 & a_pe^{i\theta} \\ a_1 & b_2 & a_2 & 0 & \cdots & 0 & 0 \\ 0 & a_2 & b_3 & a_4 & \cdots & 0 & 0 \\ \vdots & \vdots & \vdots & \vdots & \vdots & \vdots & \vdots \\ 0 & 0 & 0 & 0 & \cdots & b_{p-1} & a_{p-1} \\ a_pe^{-i\theta} & 0 & 0 & 0 & \cdots & a_{p-1} & b_p\end{matrix} \right).
\]
In particular,
\begin{equation} \lb{2.10}
      \spec(H) = \bigcup_{e^{i\theta}\in\partial\bbD} \spec(H(\theta))
\end{equation}
\end{fact}

This lets us describe band edges

\begin{fact} \lb{F2.18}
The edges of gaps correspond to eigenvalues of $H(\theta)$ for $\theta=0,\pi$, that is periodic and antiperiodic boundary conditions
\end{fact}

There is a detailed analysis; the largest periodic eigenvalue is simple and is the top of $\spec(H)$, the next two are antiperiodic  and they are unequal if and only if there is a gap with IDS value $(p-1)/p$, the next two are periodic $\dots$. One consequence of this gap edge result is

\begin{fact} \lb{F2.19}
Generically, all gaps are open, that is the set of $\{a_n, b_n\}_{n=1}^p$ in $\bbR^{2p}$ for which there is a closed gap is a closed nowhere dense set.
\end{fact}

One looks at the set in $\bbR^{2p}$ of all possible $a$'s and $b$'s for which there are gaps where the IDS is $j/p$ for all $j=1,\dots,p-1$. It is easy to see it is open and this says it is dense. In 1976, Simon \cite{Simongeneric} showed that the analog holds for continuum Schr\"{o}dinger operators.

In this Jacobi case, more is true using ideas that go back to Wigner--von Neumann \cite{WvN}.

\begin{fact} \lb{2.20}
The set of $\{a_n, b_n\}_{n=1}^p$ in $\bbR^{2p}$ where one or more gaps are closed is a real variety of codimension $2$.
\end{fact}

\medskip
\textbf{G. Isospectral Manifold}

The set of $p$ band sets, $\cup_{j=1}^p [\alpha_j,\beta_j], \quad \alpha_1<\beta_1<\alpha_2<\dots<\beta_p$, is described by $2p$ real numbers so a manifold of dimension $2p$. But they are not all possible spectra of periodic Jacobi matrices because a general set has arbitrary real harmonic measures of the bands while, in the periodic case, bands have harmonic measures of the form $j/p$.  This places $p-1$ constraints on the set (not $p$ because it suffices that $p-1$ harmonic measures be rational).

\begin{fact} \lb{F2.21}
The dimension of allowed periodic spectra of period $p$ is $p+1$, that is the set of $\{\alpha_j,\beta_j\}_{j=1}^p$ in the subset of $\bbR^{2p}$ with $\alpha_1<\beta_1<\alpha_2<\dots<\beta_p$ which is the spectrum of some period $p$ Jacobi matrix is a manifold of dimension $p+1$.
\end{fact}

\begin{fact} \lb{F2.22}
   The isospectral family associated to a $p$-band periodic spectral set is a manifold of dimension $p-1$
\end{fact}

One can describe this isospectral manifold in detail due to several beautiful underlying structures.  One involves the Toda flow and shows that the isospectral manifolds are the fibers of a completely integrable Hamiltonian system (see, for example \cite[Chapter 6]{SiSz}). In particular, they are tori. Another views the isospectral torus as the Jacobian variety of a hyperelliptic Riemann surface \cite{Netal, McvM}.

\begin{fact} \lb{F2.23}
    The isospectral family associated to a given $p$-band periodic spectral set is a torus of dimension $p-1$.
\end{fact}

\begin{fact} \lb{F2.24}
    The torus can be described by giving the position of the poles of $m_1^+$ on the two sheeted Riemann surface, one in each gap.
\end{fact}

\medskip
\textbf{H. Discriminant}

We'd be remiss if we didn't mention the discriminant, $\Delta(z)$

\begin{fact} \lb{F2.25}
 There is a polynomial, $\Delta(z)$, of degree $p$ so that
\begin{equation} \lb{2.11}
   \spec(H) = \Delta^{-1}[-2,2]
\end{equation}
\end{fact}
In the math physics literature, $\Delta$ arises as the trace of a transfer matrix (see \cite[Chapter 5]{SiSz}) while in the OP literature as a Chebyshev polynomial; see, for example \cite{GvA, Tot, CSZ1}.  This is a key tool in some proofs of the above results.

\section{Definition of Periodic Jacobi Matrices} \lb{s3}

In this section, we'll define what we mean by a periodic Jacobi matrix on a tree.  To set notation and terminology, we begin with some preliminaries and facts about graphs. Our graphs will either be finite or infinite leafless graphs with bounded degree.  References on graph theory include \cite{Bondy, Diestel,Hartsfield}.

A \emph{graph}, $\calG$, is a collection of \emph{vertices} $V\left( \calG \right)$ and of \emph{edges} $E\left( \calG\right)$.  Our edges will be undirected.  Each edge was two ends $v_1, v_2$ in $V$.  In the finite case, $V$ has $p<\infty$ elements and $E$ has $q<\infty$ elements.  We also demand that every $v\in V$ is the end of some edge.  In the infinite case, we demand that each vertex is the end of only finitely many edges.

We allow self--loops, i.e.\ edges where the two ends are the same $v\in V$ and we allow a given pair $v, v' \in V$ to be the ends of more than one edge.  The \emph{degree}, $d(v)$, of a vertex is the number of edges of which $v$ is an end, with an edge counted twice if that edge is a self--loop with $v$ at both ends.  In the infinite case, we will demand that
\begin{equation}\label{3.1}
  \sup_{v\in V} d(v) < \infty
\end{equation}

A \emph{leaf} is a vertex with degree $1$.  We will normally only consider graphs with no leaves although in the finite case, we will sometimes drop edges to produce a finite tree which must have some leaves.  A graph with constant degree is called \emph{regular}. There is a natural notion of homeomorphism between graphs and once we define trees, it is easy to see that, up to homeomorphism, there is a unique infinite regular tree of degree $d$.

While our edges are not directed, it is sometimes useful to temporarily assign a direction in which case we write $\tilde{e}_\alpha$ and call the vertices the initial and final vertices of the directed edge.  A \emph{path} in $\calG$ is a finite set of edges $e_1,\dots,e_k\in E$ and a direction for each edge so that the final vertex of $e_j$ is the initial vertex of $e_{j+1}$ for $j=1,\dots,k-1$.  We say that the path goes from the initial vertex of $e_1$ to the final vertex of $e_k$.  We say that a path is \emph{simple} if  its initial vertices are distinct, its final vertices are distinct and its edges are all distinct.  A path is called \emph{closed} (or a \emph{cycle}) if it goes from some vertex to itself.  We will only consider graphs with the property that there is a path from any vertex to any other distinct vertex.  Such graphs are called \emph{connected}.

It is easy to see that the following are equivalent under the connectedness assumption: that $\calG$ contains no simple closed paths and that there is a unique simple path between any pair of points.  In that case, we say that $\calG$ is a \emph{tree}.  Associated to any graph, $\calG$, is a topological space: We start out with a point for each vertex and a topological copy of $[0,1]$ for each edge.  We then glue the end points of the copy of $[0,1]$ associated to an edge to the two vertices at its ends.  We also use $\calG$ for this topological space.  It is easy to that $\calG$ is connected as a topological space if and only if it is connected as a graph.  Moreover, $\calG$ is a tree if and only if it is simply connected as a topological space.

It is easy to see that a finite graph which is a tree has leaves, so our basic finite graphs are never trees although we will sometimes drop some edges from our leafless graphs to get a connected tree.  A simple induction proves that any finite connected graph has $q \ge p-1$ and such a graph is a tree if and only if equality holds.  If $\calG$ is a finite connected graph and $q>p-1$, one can prove that one can turn it into a connected tree by dropping
\begin{equation}\label{3.1A}
  \ell\equiv q-(p-1)
\end{equation}
suitable edges. That implies that the fundamental group of a finite connected graph with $p$ vertices and $q$ edges is the free (non--abelian when $\ell\ge 2$) group, $\bbF_\ell$, on $\ell$ generators.

A \emph{Jacobi matrix on a graph}, $\calG$, is associated to a set of real numbers $\{b_j\}_{j\in V}$ assigned to each vertex and strictly positive reals $\{a_\alpha\}_{\alpha\in E}$ assigned to each edge. Because we will only consider finite graphs or infinite trees with periodic parameters, the sets of $a$'s and $b$'s are finite.  The Jacobi matrix acts on $\ell^2(V)$, the vector space of square summable sequences indexed by the vertices of the graph. It has matrix elements
\begin{equation} \lb{3.2}
  H_{jk} = \left\{
             \begin{array}{ll}
               b_j, & \hbox{ if $j=k$;}  \\
               \sum_\alpha a_\alpha, & \hbox{ if $j \ne k$  are ends of one or more edges } \\
               &\qquad\qquad \alpha \hbox{ which we sum over;}  \\
               0, & \hbox{ if no edges have $j$ and $k$ as ends.}
             \end{array}
           \right.
\end{equation}
If there are self--loops, one needs to modify this.

Let $\calG$ be a connected finite graph (with no leaves). Its universal cover, $\calT$, is easily seen to be the topological space associated to an infinite graph so the covering map takes edges to edges and vertices to vertices.  This graph is a tree so that if $\calG$ has constant degree, so does $\calT$, i.e. it is a \emph{regular tree}.

Now let $J$ be a Jacobi matrix on $\calG$.  There is a unique Jacobi matrix, $H$, on $\calT$ so that if $\Xi:\calT \to \calG$ is the covering map and $B_j,A_\alpha$ the Jacobi parameters of $J$ and $b_j,a_\alpha$ of $H$, then $b_j=B_{\Xi(j)}, a_\alpha=A_{\Xi(\alpha)}$.  Any deck transformation, $G\in\Gamma$, the set of deck transformations on $\calT$, induces a unitary on $\calH(\calT)$ and these unitaries all commute with $H$.  We call $H$ a \emph{periodic Jacobi matrix} and set $p$, the number of vertices of $\calG$ to be its \emph{period}, although, as we'll explain, there is some question if this is the right definition of period!

As we've discussed, if $\calG$ has $\ell$ independent cycles (equivalently, one can drop $\ell$ edges and turn $\calG$ into a connected finite tree), then the fundamental group of $\calG$ is the free nonabelian group with $\ell$ generators, $\bbF_\ell$.  So that is the natural symmetry of our periodic trees.

By the \emph{free Laplacian matrix} on a tree, we will mean the one with all $b$'s 0 and all $a$'s 1 (this is sometimes called the adjacency matrix; the Laplacian (or its negative!) has  $b_j$ equal to minus the degree at $j$.  If the tree is regular, the two differ by a constant but they don't in the non--regular case).  In this regard, there is a strange distinction between regular trees of constant degree $d$ depending on whether $d$ is even or odd!  The graph with one vertex and $\ell$ self loops has degree $d=2\ell$.  Its universal cover is the regular graph of degree $d=2\ell$ and its free Laplacian is a period $1$ Jacobi matrix.  But there is no graph with a single vertex of odd degree, so, with our definition, the free Laplacian on an odd degree homogenous tree is of period $2$! So perhaps one needs to refine our definition of period.  In any event, we'll see that there are some significant differences between periodic Jacobi matrices on homogenous trees of even and of odd degree.

The point is that the free group with $\ell$ generators acts freely (i.e. no fixed point for non-identity elements) and transitively on the degree $2\ell$ regular tree.  There is no such symmetry group on any odd degree regular tree, although by looking at the cover of the two vertex, no self loop, $d$ edge graph, one sees that $\bbF_{d-1}$ acts freely on the degree $d$ regular tree but with two orbits rather than transitively.  One can add an extra generator to get a transitive symmetry group but that action is no longer free.

Sunada \cite{Sun} who dealt primarily with continuum Schrodinger type operators noted that there is often a realization of operators invariant under discrete groups as operators acting on $\ell^2(\calT,\calW)$ where $\calT$ is a Cayley graph of the symmetry group, and $\calW$ in his situation is an infinite dimensional Hilbert space, essentially the spaces of orbits of the symmetry group.  We owe to Christiansen and Zinchenko the explicit realization of this picture for our situation which appears in Theorem \ref{T3.1}.

We use notation from appendix \ref{appendix}.  $\calT_{2\ell}$ denotes the regular tree of degree $2\ell$ thought of as the Cayley graph of $\bbF_\ell$. $\calH_{2\ell;p}$ is $\ell^2(\calT_{2\ell},\bbC^p)$.  By a \emph{simple Jacobi matrix} we mean the operator $\wti{J}$ on $\calH_{2\ell;p}$ given by
\begin{equation}\label{3.3}
  (\wti{J}u)_w=\wti{B}u_w + \sum_{j=1}^{\ell} \wti{A}_j u_{x_jw} + \sum_{j=1}^{\ell} \wti{A}_j^* u_{x_j^{-1}w}
\end{equation}
where $w$ is an index labeling a point in $\calT_{2\ell}$ thought of as an element in $\bbF_\ell$ and $\{x_j\}_{j=1}^\ell$ are the basic generators of $\bbF_\ell$.  Here $\wti{B}$ is a self--adjoint $p\times p$ matrix and $\wti{A}_1,\dots,\wti{A}_\ell$ are $p\times p$ (possibly non--self--adjoint) matrices.

Given a periodic Jacobi matrix, $H$, built over a Jacobi matrix, $J$, on a finite graph $\calG$ with $p$ vertices and $p+\ell-1$ edges, we can remove $\ell$ of those edges to get a tree, $\calT_\calG$ (aka a \emph{spanning tree} for $\calG$).  The Jacobi matrix of that tree will be $\wti{B}$ obtained by restricting the original Jacobi parameter to $\calT_\calG$  and defines a self--adjoint $p\times p$ matrix.

Let $e_\alpha$ be one of the dropped edges which we can associate with one of the generators, $x_{j(\alpha)}$, of $\bbF_\ell$.  If $a_\alpha$ is the corresponding Jacobi parameter and $e_\alpha$ goes from $i(\alpha)$ to $j(\alpha)$, we can define a rank 1 matrix, $\wti{A}_j$, with only a single non--zero matrix element, the $i(\alpha)j(\alpha)$ matrix element which has the value $a_\alpha$.  A little thought shows that

\begin{theorem}  \lb{T3.1} $H$ is unitarily equivalent to the simple Jacobi matrix on $\calH_{2\ell;p}$ with $\wti{B}$ the Jacobi matrix on $\calT_\calG$ and $\wti{A}_j$ the rank one matrices above (which are not self--adjoint if $e_\alpha$ is not a self loop).
\end{theorem}

One way of thinking of this is what we'll call the lego viewpoint of a periodic Jacobi matrix on a tree.  Our original graph, $\calG$, leads to a tree, $\calT_\calG$, (and matrix $B$ obtained by restricting $J$ to $\calT_\calG$) by throwing away $\ell$ edges $e_1,\dots,e_\ell$.  Instead of throwing away these edges, we think of cutting them leaving $2 \ell$ connectors $e_1^+,\dots,e_\ell^+,e_1^-,\dots,e_\ell^-$, half edges with only one vertex.  Our basic lego block is $\calT_\calG$ with the $2\ell$ labeled connectors sticking out.

We get the infinite Jacobi matrix, $\wti{J}$, over the Jacobi matrix $J$ on $\calG$ by placing a lego block at each vertex on the regular tree, $\calT_{2\ell}$ connecting neighboring set of lego pieces via matching connectors, an $e_j^+$ on one to an $e_j^-$ on the neighbor.  The blocks get a matrix $B$ acting on the copy of $\bbC^p$ at the corresponding site.  The connectors get an $A_j$. The finite tree $\calT_\calG$ with the dangling half edges is also known as a fundamental domain of the action of $\bbF_{2\ell}$ on $\calT$ viewed as the universal cover of $\calG$.

This realization allows certain natural *--algebras discussed in the Appendix to enter the picture.  $\calC^{(0)}_{2\ell;p}$ is the *--algebra generated by the set of all simple Jacobi matrices operating on $\calH_{2\ell;p}$, $C^*_{red}(\bbF_\ell;\bbC^p)$ its norm closure and $\calV_{2\ell;p}$ the set of all operators commuting with the natural representation of $\bbF_\ell$ on $\calH_{2\ell;p}$ -- it is the weak* closure of  $\calC^{(0)}_{2\ell;p}$.  Clearly $H$ (or the unitarily equivalent $\wti{J}$) lies in $\calC^{(0)}_{2\ell;p}$, so $f(H)$ lies in $C^*_{red}(\bbF_\ell;\bbC^p)$ if $f$ is a continuous function and in $\calV_{2\ell;p}$ if $f$ is a Borel function.  In particular

\begin{theorem} \lb{T3.2} For any $\lambda\in\bbR$, the spectral projection $E_{(-\infty,\lambda)}(H)\in\calV_{2\ell;p}$.  If $\lambda\notin\spec(H)$, then $E_{(-\infty,\lambda)}(H)\in\calC^{(0)}_{2\ell;p}$.
\end{theorem}

\section{The Free Model and the DOS} \lb{s4}

The definition we will take for the density of states (DOS), $dk$ (and so integrated density of states (IDS), $k$) is simple. We fix a finite graph, $\calG$, with $p$ vertices and $q$ edges.  For each vertex, $j\in\calG$, the spectral measure for $H$ at vertex $r\in\calT$, $d\mu_ r$ is the same for all $r\in\calT$ with $\Xi(r)=j$.  The DOS is defined by picking one $d\mu_r$ for each $j\in\calG$, summing over $j$ and dividing by $p$, the number of vertices in $\calG$.  That is
\begin{equation}\label{4.1}
  dk(\lambda) = \frac{1}{p}\sum_{j\in \calG; r \text{ so that } \Xi(r)=j} d\mu_r
\end{equation}

One of the simplest examples for which one can compute the DOS is the free Laplacian on, $\calT_d$, the regular tree of degree $d$, for which the DOS is
\begin{equation}\label{4.2}
  dk_d(\lambda) = \frac{d\sqrt{4(d-1)-\lambda^2}}{2\pi(d^2-\lambda^2)}\chi_{[-s_d,s_d]}(\lambda)d\lambda
\end{equation}
where $s_d=\sqrt{4(d-1)}$ which is the top of the spectrum.  Since it follows easily from our formalism in Section \ref{s6}, we'll compute this in Section \ref{s7}, but it has been presented earlier in discussions of random Schr\"{o}dinger operators on trees \cite{Klein, AW1,AW2} and radially symmetric potentials \cite{Den,DK,SiSz}.  It is called the Kesten--McKay distribution (after \cite{McK,Keston}) since it also describes finite random graphs of constant degree $d$ (this may seem surprising but we'll discuss the reason later).

An important tool in understanding the DOS involves the operator algebras discussed in the Appendix.  Fix a finite graph $\calG$ with universal cover tree $\calT_{2\ell}$.  As we saw in Theorem \ref{T3.2}, the spectral projections of the periodic Jacobi matrix, $H$, lie in the von Neumann algebra $\calV_{2\ell;p}$ and the spectral projection of an interval $(a,b)\subset\bbR$ with $a,b\notin\spec(H)$ lies in the $C^*$-algebra $C^*_{red}(\bbF_\ell;\bbC^p)$.  By \eqref{4.1}, we have that
\begin{equation}\label{4.3}
  k(\lambda)=p^{-1}T(E_{(-\infty,\lambda)}(H))
\end{equation}
where $T$ is the unnormalized trace of \eqref{A.4D}.

We next want to discuss whether the DOS can be viewed as an infinite volume limit of eigenvalue densities of operators restricted to finite boxes as it can in the 1D case.  Fix the base point, $e_0\in\calT_{2\ell}$ and define the ball, $\wti{\Lambda}_r$, as the set of all vertices in $\calT_{2\ell}$ with distance at most $r$ from $e_0$.  Let $\Lambda_r$ be the $p[1+\sum_{q=1}^r(2\ell)(2\ell-1)^{q-1}]$ points in $\calT$ that map to $\wti{\Lambda}_r$ under the representation of Theorem \ref{T3.1}. Because the number of boundary points in $\Lambda_r$ is comparable to the total number of points in $\Lambda_r$, one \emph{cannot} expect to get $dk$ as a limit eigenvalue counting measures with free boundary conditions.  We can make this explicit.

\begin{theorem} \lb{T4.1} Let $H$ be a periodic Jacobi matrix on any infinite tree which is not a line.  Then no limit point of free boundary condition eigenvalue counting density agrees with the DOS.  Indeed,
\begin{equation}\label{4.3A}
   \limsup_r \int \lambda^2 dN_r(\lambda) < \int \lambda^2 dk(\lambda)
\end{equation}
\end{theorem}

Here, letting $n_r=p[1+\sum_{q=1}^{r} (2\ell)(2\ell-1)^{q-1}]$ be the number of vertices in $\Lambda_r$, one defines $dN_r$ to be $(n_r)^{-1}$ times the eigenvalue counting measure for the Jacobi matrix, $J_r$, obtained by keeping only the vertices in $\Lambda_r$ and edges between them and taking the restricted Jacobi parameters associated to the resulting graph.

To prove this result, we need a graphical representation of the moments of the density of states.  If $d\mu_j$ is the spectral measure of a point $j\in\calT$, then
\begin{equation}\label{4.4}
  \int \lambda^k d\mu_j(\lambda) = \jap{\delta_j,H^k\delta_j}
\end{equation}
Expanding $H^k$, one sees that
\begin{equation}\label{4.5}
  \jap{\delta_j,H^k\delta_j}=\sum_{\omega\in W_{j,k}} \rho(\omega)
\end{equation}
where $W_{j,k}$ is the set of all ``walks'' of length $k$ starting and ending at site $j$, i.e. $\omega_1,\dots,\omega_{k+1} \in\calT$ where $\omega_1=\omega_{k+1}=j$ and for $m=1,\dots,k$, one has that either $\omega_{m+1}=\omega_m$ or $\omega_m$ and $\omega_{m+1}$ are neighbors in $\calT$ (i.e. two ends of an edge).  Moreover
\begin{equation}\label{4.6}
  \rho(\omega) = \rho_1(\omega)\dots\rho_k(\omega)
\end{equation}
\begin{equation}\label{4.7}
  \rho_m(\omega) = \left\{
                     \begin{array}{ll}
                       b_{\omega_m}, & \hbox{ if } \omega_m=\omega_{m+1} \\
                       a_{(\omega_m,\omega_{m+1})} & \hbox{ if } \omega_m\ne\omega_{m+1}
                     \end{array}
                   \right.
\end{equation}

On the other hand,
\begin{align}
  \int \lambda^k dN_r(\lambda) &= n_r^{-1} \tr(H_r^k) \nonumber \\
                               &= n_r^{-1} \sum_{j\in\Lambda_r} \jap{\delta_j,H_r^k\delta_j}  \nonumber \\
                               &= n_r^{-1} \sum_{j\in\Lambda_r} \sum_{\omega\in W_{j,k,r}} \rho(\omega) \lb{4.8}
\end{align}
where $W_{j,k,r}$ is defined like $W_{j,k}$ except that we require $\omega_m\in\Lambda_r$ instead of $\omega_m\in\calT$.

\begin{proof} [Proof of Theorem \ref{T4.1}] It clearly suffices to prove \eqref{4.3A}.  Comparing \eqref{4.5} and \eqref{4.8} and using the definition of $dk$, one sees that
\begin{align}
  \int \lambda^2 [dk-dN_r](\lambda) &= n_r^{-1} \sum_{j\in\Lambda_r} \sum_{\omega\in W_{j,2}\setminus W_{j,2,r}} \rho(\omega) \nonumber \\
                                    &= n_r^{-1} \sum_{j\in\partial\Lambda_r} \sum_{\alpha=(jk),\,k\notin\Lambda_r} a_\alpha^2 \lb{4.9}
\end{align}

Here $\partial\Lambda_r$ is the set of $j\in\Lambda_r$ which are one end of an edge in $\calT$ whose other end is not in $\Lambda_r$.  Since each boundary cell has at least one vertex in $\partial\Lambda_r$, the number of points, $s_r$, in $\partial\Lambda_r$ is at least
\begin{equation*}
  s_r \ge 2\ell(2\ell-1)^{r-1}
\end{equation*}
In particular,
\begin{equation}\label{4.10}
  \liminf_{r\rightarrow \infty}s_rn_r^{-1} \ge p^{-1}\left[ \sum_{q=0}^{\infty} (2\ell-1)^{-q} \right]^{-1} = p^{-1}[(2\ell-1)/(2\ell-2)]
\end{equation}
Therefore,
\begin{equation}\label{4.11}
\begin{split}
  \liminf_{r \rightarrow \infty}\text{RHS of }\eqref{4.9} & \ge \liminf_{r \rightarrow\infty} s_rn_r^{-1} \min(a_\alpha^2) \\
  & \ge p^{-1}[(2\ell-1)/(2\ell-2)]\min(a_\alpha^2)
\end{split}
\end{equation}
which proves \eqref{4.3A}
\end{proof}

\begin{remarks} 1. The lower bounds above are far from optimal.  For example, we have not used the fact that there are $\ell$ links from a boundary cell to the outside of $\Lambda_r$.

2.  If all $b_j\ge 0$, one can replace $\lambda^2$ in \eqref{4.3A} by $\lambda^{2k}$ for any $k\ge 1$.

3. Without much additional effort, one can actually compute the limit in \eqref{4.3A} and even prove that the measures $dN_r$ have a limit. In fact, using the \emph{canopy tree} $\mathcal{C}$ defined in \cite{AW3}, it is possible to define a simple Jacobi matrix $\widetilde{J}_\infty$ on $\mathcal{C}$ such that its appropriately defined DOS is the limit of $dN_r$. Since the details are of marginal relevance to the discussion here, we leave them to the interested reader.

4. However, this doesn't mean that there is a single possible ``free boundary condition'' density of states because in forming $\Lambda_r$ we made a choice of which links in $\calG$ to break to get $\calT_\calG$.  For example, in the $ace$ model of Example \ref{E9.4}, $\int \lambda^2 dk(\lambda)-\limsup_r \int \lambda^2 dN_r(\lambda)$ can be any of (a multiple of) $a^2+c^2$, $a^2+e^2$ or $c^2+e^2$ depending on whether we leave the $e$, $c$ or $a$ edge unbroken.
\end{remarks}

The easiest way to discuss periodic boundary conditions (BC) is to use the lego pieces picture described after Theorem \ref{T3.1}.  The boundary points in $\Lambda_r$ are described by words, $\omega=\alpha_1\dots\alpha_r$, in $\bbF_\ell$.  Here each $\alpha_j$ is one of the $\ell$ generators of $\bbF_\ell$ or its inverse with $\alpha_j\alpha_{j+1} \ne e$ for $j=1,\dots,r-1$.  We define the opposite word to be $\tilde{\omega} =\alpha_1^{-1}\dots\alpha_r^{-1}$ (this is not usually $\omega^{-1}$ which is $\alpha_r^{-1}\dots\alpha_1^{-1}$).  $\tilde{\omega}$ is obtained by walking from the origin taking the opposite step to the one taken for the walk to $\omega$.

Place a lego piece down at each vertex in $\Lambda_r$, linking neighbors by the group labels on generators on the edges.  The free boundary condition operator $H_r$ has a link at $\omega$ that uses the $\alpha_r^{-1}$ connector to $\omega$ by its only neighbor in $\Lambda_r$ but the other $2\ell-1$ connectors are unlinked.  Consider linking each connector left at $\omega$ to the matching free connector at $\tilde{\omega}$.  We define the canonical periodic boundary condition operator by taking the graph $\pi_C^{(r)}$ obtained by adding these links for each $\omega$ in $\partial\Lambda_r$, the boundary of $\Lambda_r$. $\pi_C^{(r)}$ is a finite graph with fixed degree $2\ell$ (although once we put the lego pieces in we get a finite graph which may not have a fixed degree -- recall the lego pieces, $\mathcal{T}_{\mathcal{G}}$, do not necessarily have a fixed degree).  $H_{\pi_C^{(r)}}$ is the Jacobi matrix obtained by adding to $H_r$ the $A_\alpha$ links associated to the added connectors.  There is a graphical representation for matrix elements of $H_{\pi_C^{(r)}}^k$ which involves walks that can go through the added links.  There are at least as many walks for $r$ much larger than $k$ as on the infinite graph but due to loops there can be additional walks.

$\pi_C^{(r)}$ has lots of closed loops of length $2$ since we can go from $\omega$ to $\tilde{\omega}$ by one link and come back via another.  For this reason, the eigenvalue counting measure for $H_{\pi_C^{(r)}}$ does not converge to the infinite tree DOS, indeed its second moment is strictly larger than for the infinite tree DOS.

Pairing all the links of $\omega$ to $\tilde{\omega}$ is natural in one sense. $\Lambda_r$ with lego pieces has $\frac{2\ell-1}{2\ell}s_r$ connectors of each of the $2\ell$ types, $e_1^+,e_1^-,\dots,e_\ell^+,e_\ell^-$.  By a periodic pairing, we mean a pairing of each of the $\frac{2\ell-1}{2\ell}s_r$, $e_j^+$'s with a different $e_j^-$ for each $j=1,\dots,\ell$.  There are $\left[\left(\frac{2\ell-1}{2\ell}s_r\right)!\right]^\ell$ such pairings in all.  We'll use $\calP_r$ for the set of all such pairings.  Associated to any pairing is a graph, $\pi$, of constant degree $2\ell$ and using lego pieces, also a Hamiltonian $H_\pi$ on $\ell^2(\Lambda_r,\bbC^p)$. These all have a claim to be a periodic BC operator.

There is a natural way to randomly pick a sequence of periodic pairings.  On $\calP_r$, we put counting measure (i.e.\ weight $1/\left[\left(\frac{2\ell-1}{2\ell}s_r\right)!\right]^\ell$) to each possibility).  By a random set of pairings, we means a point in $\bigtimes_{r=1}^\infty \calP_r$ of a sequence of pairings, one for each $r$ and we give the space the infinite product of counting measures.  In \cite{ABKS}, we will prove with Kalai:

\begin{theorem} \lb{T4.2} For a.e.\ sequence of periodic pairings, the eigenvalue counting measure for $H_\pi$ converges to the DOS for the infinite tree.
\end{theorem}

The reason this is true is that random $\pi$'s have few small closed loops.  Indeed we prove the expected number of loops of a fixed size is finite.  While the method of proof is different, this result is related to McKay's result \cite{McK} quoted at the start of this section that a randomly chosen degree $d$ graph has an eigenvalue density that converges to \eqref{4.2}; he shows that random graphs have few small closed loops.

There is a different way of thinking of about periodic boundary conditions.  Our graphs, $\pi$, are of constant degree $2\ell$ with labeled edges so that there is a natural map of $\pi$ to the basic graph with $p=1$ point and $q=\ell$ edges, each a self loop to the single point. By replacing the points of $\pi$ by the tree $\calT_\calG$ in the lego representation, we get a natural map onto the initial graph, $\calG$.  This map is then a covering map, i.e., the decorated $\pi$ is a covering space to $\calG$ and $H_\pi$ is the lift of the Jacobi matrix on $\calG$ to the covering space.

More generally, we can regard any such finite cover and its associated Jacobi matrix lift as a periodic BC Hamiltonian.  Such finite covers are in one-one correspondence to finite index subgroups, $\bbH$, of $\bbF_\ell$, the fundamental group of $\calG$. At first sight, this looks ideal since $\bbH$ is connected to loops in the finite covering space, but on more careful examination, one realizes that $\bbH$ is only closed loops through the base point.

An additional point is that while the periodic BC objects studied in Theorem \ref{T4.2} are reasonable from the point of view of someone who thinks about periodic BC on $\bbZ^\nu$, they have a big lack in comparison with that case. The periodic BC graph on $\bbZ^\nu$ is a torus which has a huge symmetry group while, for example, it is easy to see that the graph $\pi_C^{(r)}$ has no translational symmetries at all.  The key is to consider normal subgroups, $\bbH$, in $\bbF_\ell$.  The points in the covering space for the one point, $\ell$ self-loop case are then points in the quotient group $\bbG=\bbF_\ell/\bbH$. Because the subgroup is normal, $\bbG$ acts on it by right multiplication and the covering space has a transitive group of symmetries.  For the objects built over a graph $\calG$, the action is transitive at the level of copies of $\calG$ (although not transitive on the vertices of $\calG$ and its covering space).  So we will call the periodic BC objects associated with normal $\bbH$, \emph{symmetric periodic BC} objects.  For symmetric graphs, to know there are no small loops, it suffices to see that there are no small loops through the origin.

There is a huge literature on normal subgroups of $\bbF_\ell$ due to its importance in topology, group theory and number theory.  Using ideas from that literature, \cite{ABKS} will also prove:

\begin{theorem} \lb{T4.3} There exist sequences of symmetric periodic BC whose eigenvalue counting distributions converge to the infinite tree DOS.
\end{theorem}

\section{Sunada's Gap Labeling Theorem} \lb{s5}

The following is a fundamental result in the theory

\begin{theorem} [Sunada \cite{Sun}] \lb{T5.1} Let $H$ be a period $p$ periodic Jacobi matrix on a tree $\calT$.  Then the integrated density of states, $k(\lambda)$, takes a value $j/p,\, j\in\{1,\dots,p-1\}$ in any gap of the spectrum.  In particular, the spectrum of $H$ is a union of at most $p$ closed intervals.
\end{theorem}

\begin{remarks} 1.  It is trivial that $k$ is constant in any gap.  The deep fact is, of course, the quantization.  Fix $j\in\{1,\dots,p-1\}$.  If there is a unique point in $\spec(H)$ where $k(\lambda)=j/p$, so that this point is neither in a gap or an edge of the gap, then we say the gap at $j/p$ is closed.  Otherwise, we say that the gap at $j/p$ is open.

2.  Sunada's paper deals primarily with differential operators invariant under a discrete group but he remarks that it applies to difference operators.  Since the IDS is not normalizable in the continuum case, he only has a quantization claim.  If you follow his argument in our case, you get a precise $1/p$ quantization although it never appears explicitly in his paper.
\end{remarks}

\begin{proof}  By Theorem \ref{T3.1}, H is unitarily equivalent to an operator $\wti{J}$ on $\calH_{2\ell;p}$ where $\bbF_\ell$ is the fundamental group of the graph $\calG$ underlying the model.  By Theorem \ref{T3.2}, if $\lambda$ is in a gap of the spectrum $E_{(-\infty,\lambda)}(\wti{J})\in\calC^{(0)}_{2\ell;p}$.  If $T$ is the unnormalized trace of $\eqref{A.4D}$, then $T(E_{(-\infty,\lambda)}(\wti{J})) \in\bbZ$ by Theorem \ref{TA.3}.  By \eqref{4.3}, we conclude that $pk(\lambda)\in\bbZ$.
\end{proof}

We note a curious consequence of gap labelling:

\begin{theorem} \lb{T5.2} Let $H$ be a periodic Jacobi matrix on a tree with odd period.  Suppose that all $b_j=0$.  Then $0\in\spec(H)$
\end{theorem}

\begin{proof} Let $V$ be the operator $V\varphi(j) = (-1)^{\rho(j)}\varphi(j)$ where $\rho(j)$ is the distance of $j$ from some base vertex picked once and for all.  Then
\begin{equation}\label{5.1}
   VHV^{-1} = -H
\end{equation}
If $0\notin\spec(H)$, there is a gap $(-2\varepsilon,2\varepsilon)$ in the spectrum.  By \eqref{5.1}, we have that $k(\varepsilon) = 1-k(\varepsilon)$.  By gap labeling, for some integer $q$, $k(\varepsilon)=q/p$.  It follows that $1=k(\varepsilon)+(1-k(\varepsilon)) = 2q/p$ so $p=2q$.  Since $p$ is odd, we conclude that $0\in\spec(H)$.
\end{proof}

\section{Equations for $G$ and $M$ and the Absence of SC Spectrum} \lb{s6}

In this section, we introduce the basic Green's and $m$-functions for Jacobi operators on graphs, derive equations amongst them and use these equations to prove the absence of singular continuous spectrum for periodic Jacobi matrices on trees.  Let $H$ be a bounded Jacobi matrix on an infinite tree, $\calT$, with no leaves (for now, we do not suppose either that the Jacobi matrix or even that the tree is periodic). If $\alpha$ is an edge with ends $j$, $k$, then removing the edge $\alpha$ disconnects $\calT$ into two components,  $\calT^\alpha_j$ and $\calT^\alpha_k$, containing $j$ and $k$ respectively.  They are also infinite trees although if either vertex has degree $2$, they may have a leaf.  We let $H(\calT^\alpha_j)$ be the obvious Jacobi matrix acting on $\ell^2(\calT^\alpha_j)$ and similarly for $H(\calT^\alpha_k)$.  Define
\begin{equation} \lb{6.1}
  G_j(z) = \jap{\delta_j,(H-z)^{-1}\delta_j} \qquad m^\alpha_j = \jap{\delta_j,(H(\calT^\alpha_j)-z)^{-1}\delta_j}
\end{equation}
and similarly for $m^\alpha_k$.  These are defined as analytic functions on $\bbC\setminus [A,B]$ if $A$ and $B$ are the bottom and top of $\spec(H)$.  They are also analytic at infinity and in the gaps of suitable spectra. In the periodic case, one can show that the three operators have the same essential spectra, so all are meromorphic on $\bbC\setminus\text{\rm{ess spec}}(H)$.

We want to derive the equations for $G$ and $m$. These have often appeared in the literature on random discrete Schr\"{o}dinger operators on regular trees; see e.g. Klein \cite{Klein}, Aizenman-Sims-Warzel \cite{AW1,AW2} and Froese-Hasler-Sptizer \cite{FHS1,FHS2}. Unfortunately these equations have appeared many times with incorrect signs (whose corrections, fortunately, don't seem to effect the validity of their theorems)!  So we want to provide a derivation. A particularly clean method involves Banachiewicz' formula \cite{Banach} from the theory of Schur complements \cite{Schur} (also known as the method of Feshbach projections, after \cite{Fesh}, or the Liv\v{s}ic matrix after \cite{Livsic}) . This formula has been used often before in spectral theory in related contexts, see, for example, \cite{BFS, Howland}. Consider a Hilbert space that is a direct sum $\calH=\calH_1\oplus\calH_2$ so that any $N\in\calL(\calH)$ can be written
\begin{equation*}
  N = \left(
        \begin{array}{cc}
          X & Z \\
          Z^* & Y \\
        \end{array}
      \right)
\end{equation*}
where, for example, $X\in\calL(\calH_1)$.  Given such an $N$ with $Y$ invertible, we define the \emph{Schur complement} of $Y$ as $S = X-ZY^{-1}Z^*\in\calL(\calH_1)$. Let
\begin{equation*}
  L = \left(
        \begin{array}{cc}
          \bdone & 0 \\
          -Y^{-1}Z^* & \bdone \\
        \end{array}
      \right) \mbox{ so } L^{-1} = \left(
                                     \begin{array}{cc}
                                       \bdone & 0 \\
                                       Y^{-1}Z^* & \bdone \\
                                     \end{array}
                                   \right)
\end{equation*}

A simple calculation shows that
\begin{equation}\label{6.2}
  L^*NL = \left(
               \begin{array}{cc}
                 S & 0 \\
                 0 & Y \\
               \end{array}
             \right)
\end{equation}
so that
\begin{align*}
  N^{-1} &= L\left(
              \begin{array}{cc}
                S^{-1} & 0 \\
                0 & Y^{-1} \\
              \end{array}
            \right) L^* \\
           &= \left(
                        \begin{array}{cc}
                          S^{-1} & -S^{-1} ZY^{-1} \\
                          -Y^{-1} Z^*S^{-1} & Y^{-1}+Y^{-1}Z^*S^{-1}ZY^{-1} \\
                        \end{array}
                      \right)
\end{align*}
which proves Banachiewicz' formula
\begin{equation}\label{6.3}
(N^{-1})_{11} = S^{-1}
\end{equation}

For a tree, we fix $j\in\calT$ so that we can write $\ell^2(\calT) = \bbC\oplus\ell^2(\cup_{\alpha=(jk)} \calT_k^\alpha)$ corresponding to singling out the site $j$. Then $(N^{-1})_{11}$ is a number, $X$ is $b_j$, $Y =\oplus_{\alpha=(jk)} H(\calT_k^\alpha)$ and $Z$ is the various $a_\alpha$. The result of applying \eqref{6.3} both to $H$ and to $H(\calT_j^{(rj)})$ is

\begin{theorem} \lb{T6.1}  Let H be an bounded Jacobi matric on an infinite tree $\calT$ (not necessarily periodic).  Then the Green's functions and $m$-functions given by \eqref{6.1} are related by

\begin{equation} \lb{6.4}
  G_j(z) = \frac{1}{-z+b_j-\sum_{\alpha=(jk)} a^2_\alpha m_k^\alpha(z)}
\end{equation}
If $\beta=(rj)$ is an edge in $\calT$, we have that
\begin{equation} \lb{6.5}
  m^\beta_j(z) = \frac{1}{-z+b_j-\sum_{\alpha=(jk);\,k\ne r} a^2_\alpha m_k^\alpha(z)}
\end{equation}
\end{theorem}

Note that if $H$ is a periodic Jacobi matrix built over a finite graph $\calG$ with $p$ vertices and $q$ edges, the Green's and $m$-functions are invariant under the action of the group $\bbF_\ell$ so that there are $p$ distinct Green's functions and $2q$ $m$--functions.

If you compare the two equations for $G$ and $m$, they differ in a single term, so returning to the not necessarily periodic case

\begin{corollary} \lb{C6.2}  If $\beta=(rj)$ is an edge in $\calT$, we have that
\begin{equation} \lb{6.6}
  G_j(z) = \frac{1}{\left[m_j^\beta(z)\right]^{-1} - a_\beta^2 m_r^\beta(z)}
\end{equation}
\end{corollary}

This is an analog of a well known formula from the 1D case; see, for example, Simon \cite[(5.4.45)]{SiSz}.

We should mention that in \cite{AomotoEqn}, Aomoto derived some coupled equations for Green's functions (without being explicit about what branches of the square roots one needs to take):
\begin{equation}\label{6.6A}
  G_j(z)^{-1} = z-b_j-\tfrac{1}{2}\sum_{\alpha=(jk)} \left(-G_j(z)^{-1}+\sqrt{G_j(z)^{-2}+4a_\alpha^2(G_k(z)/G_j(z))}\right)
\end{equation}

Given the definition of spectral measures and \eqref{6.1}, we have that
\begin{equation}\label{6.7}
  G_j(z) = \int \frac{d\mu_j(\lambda)}{\lambda-z}
\end{equation}
so that in the periodic case
\begin{equation}\label{6.8}
 \frac{1}{p} \sum_{\text{one }j\text{ in each orbit}}G_j(z) = \int \frac{dk(\lambda)}{\lambda-z}
\end{equation}

That means (given that we will see below that $d\mu_j$, and so $dk$, have no singular continuous part) that one can recover $dk$ from knowing all the $G_j$: pure points of $dk$ occur at the poles of $G_j$ and the pure point masses are essentially averages of the residues at these poles and the a.c.\ weight of $k$ is related to limits of averages of $G_j(\lambda+i\epsilon)$.  In the next section, we'll use Theorem \ref{T6.1} to compute $G$ and $m$ in certain models and then compute $dk$ using this connection.

In the periodic case, we'll see that the Green's and $m$-functions are algebraic by which we mean functions $f(z)$ that solve $P(z,f(z))=0$ where $P$ is a polynomial in two variables (that depends non--trivially on both variables).  We need several well known results about algebraic functions.  By the degree in $w$ of $P(z,w)$ we mean the highest power of $w$ that occurs with a non-vanishing coefficient.  First we need the following:

\begin{theorem} \lb{T6.3} Let $P(z,w)$ be a polynomial of degree $n$ in $w$.  Then there exist two finite sets $F_1$ and $F_2$ in the Riemann sphere, $\widehat{\bbC}$, and $k\le n$, so that if $z_0\notin F\equiv F_1\cup F_2$, then $P(z_0,w)=0$ has $k$ solutions, each the value of a function $f_j(z)$ analytic near $z_0$. Each $f_j$ can be meromorphically continued along any curve in $\left(\widehat{\bbC} \right)\setminus F_1$ with the only possible poles at $F_2$. At points in $F$, there are fewer than $k$ solutions. At points $z_2\in F_2$, there are $k$ functions meromorphic at $z_2$ which give all solutions for $z$ near $z_2$ while at points $z_1\in F_1$ all the solutions near $z_1$ are given by one or more Laurent-Pusieux series.
\end{theorem}

\begin{remarks} 1.  Thus the functions defined at regular points (i.e.\ points not in $F$) are analytic functions which can be meromorphically continued to multisheeted Riemann surfaces with branch points at (perhaps a proper subset of) $F_1$ with possible poles at $F_2$.  The branch points are at most $k$ fold and there are at most $k$ sheets.

2. This follows, for example, from \cite[Theorem 3.5.2]{BCA}; there the set is discrete but by looking at $Q(z,w)=z^nP(1/z,w)$, one sees it is also discrete at infinity, so $F$ is finite.  This remark also shows that an algebraic function of $1/z$ is algebraic in $z$.
\end{remarks}

By algebraic functions, we'll mean either the global possibly multisheeted functions or else the local germs defined by a single solution near a point of analyticity. Since the analytic continuation of a function obeying $P(z,f(z))=0$ near a point $z_0$ obeys the same equation near any point to which $f$ can be analytically continued, the two notions are essentially the same. In one place, (to conclude that $G$ is algebraic; we could avoid it but it is more elegant to prove it) we will need:

\begin{theorem} \lb{T6.4} The set of algebraic functions is a field, that is sums, products and inverses of such functions are again such functions.
\end{theorem}

\begin{remarks}  1.  One proof of this, attractive to algebraists, uses the fact that the algebraic functions are precisely the set of solutions of polynomials over the field of rational functions and the fact \cite[Proposition V.1.4]{Lang} that the algebraic closure of a field is a field.

2.  Another proof, attractive to analysts, that uses tensor products of matrices, can be obtained by modifying \cite[Proposition III.4.4]{SimonGroup}. That proposition proves that the algebraic integers are a ring, but by replacing integers by polynomials in $z$, it shows that the algebraic functions are a ring.  To see that the inverse of an algebraic function is also an algebraic function, one notes that if $P(z,w)$ is a polynomial then, for suitable $k$, so is $R(z,w) = w^kP(z,1/w)$ and if $P(z,f(z))=0$, so does $R(z,1/f(z))=0$.
\end{remarks}

The final general theorem on algebraic functions that we'll need can be found in Lang \cite[Proposition VIII.5.3]{Lang}:

\begin{theorem} \lb{T6.5} Let $\{P_j(z,w_1,\dots,w_n)\}_{j=1}^n$ be $n$ polynomials in $n+1$ variables.  Suppose that $(z_0,w_1^{(0)},\dots,w_n^{(0)})$ is a point where
\begin{equation}\label{6.9}
  \left.
    \begin{array}{ll}
     P_j(z_0,w_1^{(0)},\dots,w_n^{(0)})\}=0,\,j=1,\dots,n \\
     \det\left(\frac{\partial P_j}{\partial w_j}\right)_{j,k=1,\dots,n}(z_0,w_1^{(0)},\dots,w_n^{(0)})\ne 0
    \end{array}
  \right.
\end{equation}
Then there is a neighborhood, $N$, of $z_0$, and $\delta>0$ so that for $z\in N$, there is a unique solution, $f_j(z), \,j=1,\dots,n$ of $P_j(z,f_1(z),\dots,f_n(z))=0\,j=1,\dots,n$ with $|f_j(z)-w_j^{(0)}| < \delta,\,j=1,\dots,n$.  Moreover, each $f_j$ is an algebraic function.
\end{theorem}

Everything before ``Moreover'' is just the implicit function theorem for analytic functions \cite[Theorem 1.4.11]{Krantz}.  That the functions are algebraic is the essence of what is in Lang's text.  We can now apply these theorems to Green's and $m$-functions.

\begin{theorem} \lb{T6.6} The Green's functions and $m$--functions of a periodic Jacobi matrix on a tree, defined originally on $\bbC\setminus [A,B]$, are algebraic functions.  In particular, there is a finite subset, $F_0 \subseteq \bbR$, so that uniformly on compact subsets of $\bbR\setminus F_0$, all these functions have limits (which might be equal to infinity) evaluated at $x+i\varepsilon$ as $\varepsilon\downarrow 0$.  These functions all have meromorphic continuations to a finite sheeted Riemann surface.
\end{theorem}

\begin{proof}  In terms of $u=1/z$ and $f_j^\beta(u) = m_j^\beta(1/u)$, the equations \eqref{6.5} for the $2q$ functions $f_j^\beta$ can be written (with $\beta=(rj)$:
\begin{equation}\label{6.10}
  f_j^\beta(u)+ u -b_j u f_j^\beta(u) + \sum_{\alpha=(jk);k\ne r} u a_\alpha^2 f_j^\beta(u) f_k^\alpha(u) = 0
\end{equation}
These equations hold at $u=0, f_j^\beta = 0$ and at that point, the Hessian matrix whose determinant appears in \eqref{6.9} is the identity, so the determinant is $1$.  It follows by Theorem \ref{T6.5} that for $u$ small there are unique small solutions $f$'s which are algebraic in $u$.  We know the $m$'s go to zero at infinity, so the $f$'s they define must be this unique small solution which is algebraic.  It follows that the $m$'s are algebraic in $z$.  Since each $m$ is algebraic, by Theorem \ref{T6.3}, they have analytic continuations to multisheeted analytic functions over $\bbC\setminus F_j^\beta$.  So the collection of $m$'s has a continuation to a mutlisheeted family over $\bbC\setminus F$ with $F=\cup_{j,\beta} F_j^\beta$.  By analyticity, this family obeys the set of polynomial equations \eqref{6.5}.

\eqref{6.4} or \eqref{6.6} and Theorem \ref{T6.4} then imply that the G's are algebraic.  The remaining assertions follow from Theorem \ref{T6.3} if one uses the fact that the $m$'s are initially given as analytic functions on $\bbC\setminus [A,B]$.  Thus while their analytic continuation might have singularities at possible non-real points in $F$, these non-real singularities are on different sheets so that the only singularities on the initial sheet lie on $F_0 = F\cap\bbR$.
\end{proof}

We note that given that the basic equations are quadratic in the $f$'s and that there are $2q$ of them, Bezout's Theorem \cite[Section 3.2.2]{Shar} implies that the number of sheets is at most $2^{2q}$, although in the few cases we can compute, there are only $2$ sheets.  Theorem \ref{T6.6} has one immediate corollary that is so important we'll call it a Theorem:

\begin{theorem} \lb{T6.7} Periodic Jacobi matrices on trees have no singular continuous spectrum.
\end{theorem}

\begin{remarks} 1. This result has not appeared previously and we regard it as the most important result in this paper.  That said, it is an immediate consequence of the fact that the $G_j(z)$ are algebraic.  In \cite{AomotoPoint}, Aomoto claims that the Green's functions are algebraic although he makes no mention at all of singular continuous spectrum.  His argument that these functions are algebraic depends on his equations \eqref{6.6A}.  They are not polynomial but are what he calls algebraic.  He remarks that these are $p$ algebraic equations in $p$ unknowns and such functions are algebraic.  He is thus relying on some unstated (but probably correct) extension of Theorem \ref{T6.5}.  More importantly, his argument is incomplete since he doesn't prove that these equations are independent, essentially some kind of condition on the invertibility of a Hessian matrix.

2. This is an analog of Fact \ref{F2.6}.  One can ask about Fact \ref{F2.7}.  In the next section we'll see that there are examples with point spectrum and in the section after, we'll discuss the Theorem of Aomoto \cite{AomotoPoint} that periodic Jacobi matrices on regular trees have no point spectrum.
\end{remarks}

\section{Several Simple Examples} \lb{s7}

Using discriminants (Fact \ref{F2.25}), you can compute closed forms for virtually any 1D periodic Jacobi matrices (up to the solution of a high degree polynomial in two variables).   There seem to be only are a few examples where one can do the same for periodic trees.  Since we've found these examples illuminating, we will discuss some here and in the next few sections.

\begin{example} \lb{E7.1} (Degree d regular tree) Let $\calT_d$ be the homogenous tree of degree $d$ and $H$ the Jacobi matrix with all $b=0$ and all $a=1$.  \eqref{6.5} yields
\begin{align}
  m(z) = \frac{1}{-z-(d-1)m(z)} &\Rightarrow (d-1)m^2+zm+1=0 \nonumber \\
                                &\Rightarrow m(z) = \frac{-z+\sqrt{z^2-4(d-1)}}{2(d-1)} \lb{7.1}
\end{align}
We want the branch of the square root which is $\text{O}(\tfrac{1}{z})$ near $\infty$ (at least on the principal sheet) since near $\infty$, $m\to 0$ rather than $m\to\infty$.

Using \eqref{6.4}, we find that
\begin{align}
  G(z) &= \frac{1}{-z-dm(z)} \nonumber \\
       &= \frac{2(d-1)}{(2-d)z-d\sqrt{z^2-4(d-1)}} \lb{7.2}
\end{align}
Multiplying numerator and denominator by $(2-d)z+d\sqrt{z^2-4(d-1)}$ and using $[(2-d)z]^2+[d\sqrt{z^2-4(d-1)}]^2=4(d-1)(d^2-z^2)$, we get
\begin{equation}\label{7.3}
  G(z)=\frac{(2-d)z+d\sqrt{z^2-4(d-1)}}{2(d^2-z^2)}
\end{equation}

The boundary value of $\text{Im}(G)$ on the real axis comes from the square root and is non-zero only when its argument is negative so we get what appeared as \eqref{4.2} ($s_d=\sqrt{4(d-1)}$).
\begin{equation}\label{7.4}
    dk_d(\lambda) = \frac{d\sqrt{4(d-1)-\lambda^2}}{2\pi(d^2-\lambda^2)}\chi_{[-s_d,s_d]}(\lambda)d\lambda
\end{equation}

Returning to $G$, it may appear that it has a pole at $z=\pm d$ and so $\pm d$ an eigenvalue but (at least on the principal sheet) the numerator also vanishes there, so $G$ is not singular on the principal sheet at those points.  This phenomenon is connected to the fact that $Hu=du$  has a positive solution (namely $u\equiv 1$) which is not in $\ell^2$.

We note that in the sense of \cite{JB}, this model is spherically symmetric, so the computation of $m$ can be reduced to the half-line $m$ functions associated to a conventional Jacobi matrix (see \eqref{7.5} below).
\end{example}

\begin{example}  \lb{E7.2}  Fix two integers $r$ and $g$, each at least 2.  This model will have an underlying finite graph, $\calG$, with $p=r+g$ vertices which we think of as $r$ red vertices and $g$ green vertices.  $\calG$ has $q=rg$ edges - specifically, every red vertex has degree $g$ and is connected to every green vertex by a single edge.  Thus each green vertex has degree $r$.  There are no edges between vertices of the same color. This model was introduced by Aomoto \cite{AomotoPoint} because, as we will see, this Jacobi matrix has point spectrum when $r\ne g$ (illuminating partially how much Fact \ref{F2.7} extends to general tree).  Aomoto mentions that the existence of point spectrum in this model follows from point spectrum in a more complicated model studied in \cite{BK, Co, Fa, Ku}.

While Aomoto does not write down explicit eigenfunctions, one can.  To see this, one writes the model as a radial tree. Consider a tree with a single red vertex at level $0$, linked (i.e.\ connected by an edge) to $g$ green vertices at level $1$.  Each of these green vertices is also linked to $r-1$ red vertices at level $2$ and, in turn, these vertices are also linked to $g-1$ green vertices at level $3$.  In general, level $2k-1;\,k=1,2,\dots$ has $g[(r-1)(g-1)]^{k-1}$ green vertices while level $2k;\,k=1,2,\dots$ has $g(r-1)[(r-1)(g-1)]^{k-1}$ red vertices.  Each level $m$ vertex, $m=1,2,\dots$ has one edge linking it to a level $m-1$ vertex and either $r-1$ or $g-1$ edges linking to level $m+1$ vertices. It is easy to see this tree is the universal cover to the basic finite red-green graph described at the start.

We note in passing, that using the formalism in \cite{JB}, the $m$-function at the central point and radial eigenfunctions can be studied with the classical tri-diagonal Jacobi matrix:
\begin{equation} \lb{7.5}
  J = \left(
        \begin{array}{ccccc}
          0 & \sqrt{g} &   &  & \dots \\
          \sqrt{g} & 0 & \sqrt{r-1} &  & \dots \\
            & \sqrt{r-1} & 0 & \sqrt{g-1} & \dots \\
            &  & \sqrt{g-1} & 0 & \sqrt{r-1} \\
          \dots  &\dots  & \ddots & \ddots & \ddots \\
        \end{array}
      \right)
\end{equation}
We will instead use direct calculation of eigenfunctions and the formulae for $m$- and Green's functions in Section \ref{s6}.  We want a radially symmetric (i.e.\ constant on each level) solution of $Hu=0$.  We'll indicate the value at level $j$ as $u_j$ and suppose that $u_0=1$.  That $Hu$ vanishes at level $0$ implies that $u_1=0$ and that $Hu$ vanishes at level $1$ implies that $u_2 = -1/(r-1)$.  A simple induction shows that
\begin{equation}\label{7.6}
  u_{2j+1} = 0 \qquad u_{2j} = (-1/(r-1))^j
\end{equation}

That the eigenfunction be $\ell^2$ is equivalent to $\sum_{j=0}^{\infty} N_j|u_j|^2 < \infty$ where $N_j$ is the number of vertices at level $j$ which we computed above.  This is equivalent to
\begin{equation}\label{7.7}
  \sum_{j=1}^{\infty} \left[\frac{g-1}{r-1}\right]^j < \infty
\end{equation}
which happens if and only if $r > g$.  There is also a square integrable eigenfunction when $g > r$ but radially symmetric in a form that the tree is written as one centered at a green vertex.  We'll see below that when $r > g$, all $\ell^2$ eigenfunctions vanish at the green sites.  In \cite{CSZinprep}, Christiansen et al prove that linear combinations of the radially symmetric functions about each of the red vertices (which are linearly independent but not orthogonal) are dense in the eigenspace.

Let $m_r$, resp.\ $m_g$, be the $m$-function at the red, resp.\ green, vertices (i.e.\ $m_r=m_r^{(rg)})$.  Then \eqref{6.1} says that
\begin{equation}\label{7.8}
  m_r(z) = \frac{1}{-z-(g-1)m_g(z)} \qquad m_g(z) = \frac{1}{-z-(r-1)m_r(z)}
\end{equation}

Substituting one equation in the other yields quadratic equations whose solutions involve a function
\begin{equation}\label{7.9}
  \Phi(z) = z^4+2(2-(r+g))z^2+(r-g)^2
\end{equation}
in terms of which the solutions are
\begin{equation}\label{7.10}
  m_r(z) = \frac{(g-r)-z^2+\sqrt{\Phi(z)}}{2(r-1)z} \qquad m_g(z) = \frac{(r-g)-z^2+\sqrt{\Phi(z)}}{2(g-1)z}
\end{equation}
where the root of $\Phi$ is taken which is $+z^2$ near $\infty$ on the principle sheet (this branch is chosen so that the $m$-functions go to $0$ at $\infty$).

By \eqref{6.4}, we find that the Green's functions at red and green vertices are given by:

\begin{align}
  G_r(z) &= \frac{1}{-z-g m_g(z)} \lb{7.11} \\
         &= \frac{2(g-1)z}{\left((2-g)z^2-g(r-g)-g\sqrt{\Phi(z)}\right)} \lb{7.12} \\
         &= \frac{(2-g)z^2-g\left[(r-g)-\sqrt{\Phi(z)}\right]}{2rgz-2z^3} \lb{7.13}
\end{align}
where we get \eqref{7.13} by multiplying numerator and denominator by $((2-g)z^2-g(r-g)+g\sqrt{\Phi(z)})$ and using that
\begin{equation*}
  [(2-g)z^2-g(r-g)]^2+g^2\Phi(z) = 2(g-1)z[2rgz-2z^3]
\end{equation*}

Similarly, we get
\begin{equation}\label{7.14}
  G_g(z) = \frac{(2-r)z^2-r\left[(g-r)-\sqrt{\Phi(z)}\right]}{2rgz-2z^3}
\end{equation}

In analyzing these equations, we take $r > g > 2$ without requiring them to be integral (later we'll also consider $r=g$). Write $\Phi(z) = Q(z^2)$ where $Q$ is a quadratic polynomial whose discriminant is
\begin{equation}\label{7.15}
  \Delta = [2(r+g-2)]^2-4(r-g)^2 = 16rg+16-16(r+g) = 16(r-1)(g-1) > 0
\end{equation}
Thus $Q$ has two positive roots (where $s=r-1, t=g-1$):
\begin{align}
  \gamma_\pm = \frac{1}{2}\left[2r+2g-4 \pm \sqrt{\Delta}\right] &= s+t \pm 2\sqrt{st} \label{7.16} \\
                                                                 &= (\sqrt{s}\pm\sqrt{t})^2 \nonumber
\end{align}
Thus $\Phi(x)$ on the real axis is negative on $(-\sqrt{\gamma_+},-\sqrt{\gamma_-})\cup(\sqrt{\gamma_-},\sqrt{\gamma_+})$ and positive on $(-\infty,-\sqrt{\gamma_+})\cup(-\sqrt{\gamma_-},\sqrt{\gamma_-})\cup(\sqrt{\gamma_+},\infty)$.  Since we take the branch of $\sqrt{\Phi}$ which is positive near $\pm\infty$ on $\bbR$ and negative on the imaginary axis, we conclude that $\sqrt{\Phi(x)}$ is positive on $(-\infty,-\sqrt{\gamma_+})\cup(\sqrt{\gamma_+},\infty)$ and negative on $(-\sqrt{\gamma_-},\sqrt{\gamma_-})$.

The denominator in \eqref{7.13} and \eqref{7.14} vanishes at $z=0$ and $z=\pm\sqrt{rg}$.  Notice that $rg=(s+1)(t+1)=s+t+(1+st)>\gamma_+$, so we know that $\sqrt{\Phi(\sqrt{rg})}$ is the positive square root. Since $\Phi(\pm\sqrt{rg}) = (rg-g-r)^2$, and $rg-r-g=(r-1)(g-1)-1>0$, we conclude that $\sqrt{\Phi(\sqrt{rg})}=rg-r-g$ which implies that the numerators of \eqref{7.13} and \eqref{7.14} vanish and there are no principle sheet poles of $G$ at $z=\pm\sqrt{rg}$.  On the other hand since $\sqrt{\Phi(0)}$ by the above is negative, we conclude that $\sqrt{\Phi(0)}=g-r$.  It follows that the numerator of \eqref{7.14} vanishes at $z=0$ and so $G_g$ has no poles in the spectral region, so, in particular, all the eigenfunctions with eigenvalue $0$ vanish at the green sites.  On the other hand, the numerator of \eqref{7.13} at $z=0$ is $-2g(r-g)$ so we have that near $z=0$.
\begin{equation}\label{7.17}
  G_r(z) = -\frac{r-g}{rz}+ \text{O}(1)
\end{equation}

Next, by \eqref{7.13} and \eqref{7.14}, the Stieltjes transform of the DOS is
\begin{equation}\label{7.18}
  \frac{rG_r(z)+gG_g(z)}{r+g} = \frac{(r+g-rg)z^2+rg\sqrt{\Phi(z)}}{z(r+g)(rg-z^2)}
\end{equation}
from which we get (taking into account that $r$ times the residue in \eqref{7.17} is $-(r-g)$) that the DOS measure has a point mass of weight $\frac{r-g}{r+g}$ at zero and an a.c.\ weight on $(-\sqrt{\gamma_+},\sqrt{\gamma_-})\cup(\sqrt{\gamma_-},\sqrt{\gamma_+})$ of
\begin{equation}\label{7.19}
  \frac{dk}{d\lambda}(\lambda)=\frac{rg\sqrt{-\Phi(\lambda)}}{(\pi\lambda)(r+g)(rg-\lambda^2)}
\end{equation}
Since the total mass of the DOS is $\frac{r+g}{r+g}$ and the model is symmetric about $0$, in line with gap labeling, we see each of the ac bands has DOS total weight $\frac{g}{r+g}$.

Finally, we note that if $r=g=d$ the radial picture shows that the Jacobi matrix is the same as the degree $d$ regular tree and \eqref{7.19} yields \eqref{7.4}.  By \eqref{3.2}, the graph $\calG$ of this model has $rg-r-g+1 = (r-1)(g-1)$ loops.  When $d=2\ell$ is even, we can write the $d$-regular tree as a period one operator as the universal cover of a finite graph with $\ell$ loops and, for general $d$ as a period $2$ operator as the universal cover of a graph with $d$ edges.  The period $2d$ graph of this model is just a finite cover of these simpler graphs.
\end{example}

\begin{example} \lb{E7.4} This is the simplest period $2$ example and confirms the idea that for non--constant $b$, all gaps are open.  The graph $\calG$ has $p=2$ points, and so two different values of the parameter $b$ that we set to be $b>0$ and $-b$. There are $q=d$ edges joining the vertices with all $a=1$.  The universal cover is $\calT_d$, the regular tree of degree $d$ and $H$ is the sum of the adjacency matrix of this tree and the diagonal matrix with values $b$ and $-b$ on ``alternate'' vertices.  By replacing $H$ by $\gamma H+ \eta\bdone$, one can describe the model where all the edges have a common value and the two vertices any pair of values.

There are two Green's functions, $G_+$ and $G_-$ for the vertices that have diagonal value $+b$ and $-b$.  The $m$-functions concern a rooted tree where the root only has degree $d-1$.  We use $m_\pm$ when the root has diagonal value $\pm b$.  Then \eqref{6.5} becomes
\begin{equation}\label{7.20}
  m_\pm(z) = \frac{1}{-z\pm b - (d-1) m_\mp(z)}
\end{equation}
while \eqref{6.4} becomes
\begin{equation}\label{7.21}
  G_\pm(z) = \frac{1}{-z\pm b - d m_\mp(z)}
\end{equation}

The equation \eqref{7.20} for $m_+$ has $m_-$ on the right but one can use the $m_-$ equation from \eqref{7.20} to obtain
\begin{equation}\label{7.22}
  m_+ = \frac{(b+z)+(d-1)m_+}{b^2-z^2+(d-1)(b-z)m_++(d-1)}
\end{equation}
or
\begin{equation}\label{7.23}
  (d-1)(b-z)m_+^2 + (b^2-z^2)m_+-(z+b) = 0
\end{equation}

Solving this equation and the similar one for $m_-$, we obtain
\begin{equation}\label{7.24}
  m_\pm = \frac{z^2-b^2-\sqrt{(z^2-b^2)^2-4(d-1)(z^2-b^2)}}{(d-1)(-z\pm b)}
\end{equation}

The discriminant is $\Delta = (z^2-b^2)^2-4(d-1)(z^2-b^2)$, so $\Delta = 0$ at $z=\pm b,\pm\sqrt{b^2+4(d-1)} \equiv \pm c$ so we expect that the spectrum of $H$ is $[-c,-b] \cup [b,c]$.

Inserting the formula for the $m$-functions in \eqref{7.21}, we get that
\begin{equation} \label{eq:GreenEx}
G_\pm(z)=\frac{(d-2)\left(b^2-z^2\right)+ d \sqrt{\Delta}}{2(z\mp b)\left(d^2-z^2+b^2 \right)}
\end{equation}
The Stieltjes transform of the DOS is
\begin{equation} \label{eq:DOS1Ex}
\frac{G_++G_-}{2}=\frac{z\left((d-2)\left(b^2-z^2 \right)+d\sqrt{\Delta}\right)}{2\left(d^2-z^2+b^2 \right)\left( z^2-b^2 \right)}
\end{equation}
so, by taking the limit of the imaginary part, we get that
\begin{equation} \label{eq:DOS2Ex}
\frac{dk}{d\lambda}(\lambda)=\frac{ |\lambda| d \sqrt{\left(\lambda^2-b^2 \right)\left(c^2-\lambda^2 \right)}}{2\pi\left(c^2+(d-2)^2-\lambda^2 \right)\left(\lambda^2-b^2 \right)}\chi_{[-c,-b] \cup [b,c]}(\lambda)
\end{equation}
which reduces to the Kesten-Mckay distribution when $b=0$ and to the one-dimensional result for $d=2$.
\end{example}

\section{Aomoto's Index Theorem} \lb{s8}

In \cite{AomotoPoint}, Aomoto proved a remarkable result about point spectra of periodic Jacobi matrices on trees.  Fix $\lambda\in\bbR$ and suppose that $Hu=\lambda u$ has $\ell^2$ solutions for a periodic Jacobi matrix built over a graph $\calG$.  Since the eigenspace, $\mathrm{Im} (P_{\{\lambda\}})$, is invariant under the group action and the unitaries in the group action go weakly to $0$ as the group parameter leaves compact sets, this eigenspace is infinite dimensional.  Aomoto defines $X^{(1)}_\lambda$ to be the set of $x\in\calG$ for which there is $y\in\calT$ with $\Xi(y)=x$ and some eigenfunction $u$ with eigenvalue $\lambda$ and $u(y)\ne 0$.  These are precisely the set of $x$ so that the Green's function $G_x(z)$ has a pole at $\lambda$.

If follows from the arguments in \cite{AomotoPoint} that $X^{(1)}_\lambda$ cannot contain any cycles or self-loops (Aomoto actually considers more general isomorphism groups that may have nontrivial elements of finite order and so needs to consider also the case of self-loops in $X^{(1)}_\lambda$). Therefore $X^{(1)}_\lambda$ is always a proper subset of $\calG$.  Let $p\left(X^{(1)}_\lambda \right)$ be the number of points in $X^{(1)}_\lambda$ and $q\left(X^{(1)}_\lambda \right)$ the number of edges of $\calG$ whose end points are both in $X^{(1)}_\lambda$.  Let $X^{(-1)}_\lambda$ be the set of vertices of $\calG$ that are not in $X^{(1)}_\lambda$ but are connected to some point in $X^{(1)}_\lambda$ by an edge in $\calG$ and let $p\left(X^{(-1)}_\lambda \right)$ be the number of points in it. Aomoto proves

\begin{theorem} [Aomoto's Index Theorem \cite{AomotoPoint}] \lb{T8.1} Let $T$ be the unnormalized trace discussed in the Appendix and suppose $\lambda$ is an eigenvalue of $H$.  Then
\begin{equation}\label{8.1}
  T\left(P_{\{\lambda\}} \right) = p\left(X^{(1)}_\lambda \right)-q\left(X^{(1)}_\lambda \right)-p\left(X^{(-1)}_\lambda \right)
\end{equation}
\end{theorem}

\begin{remark}
As noted above, the main result in \cite{AomotoPoint} includes the case where there are self-loops in $X^{(1)}_\lambda$, and in that case the right hand side of \eqref{8.1} is a little different.
\end{remark}

While we can follow Aomoto's proof, it is involved and we don't understand why it works.  Moreover, while the striking fact is that the right side is an integer, the proof has many intermediate formulae with non-integral terms.  So we raise

\begin{problem} \lb{P8.2} Find a natural, easy to understand, proof of Theorem \ref{T8.1}
\end{problem}

Gap labeling implies that if $\lambda$ is an isolated point of $\spec(H)$ with $H$ a period $p$ Jacobi matrix on a tree then the DOS measure gives weight to $\{\lambda\}$ of the form $j/p$ with $j$ a positive integer.  Theorem \ref{T8.1} implies the more general result

\begin{corollary} \lb{C8.3} Let $\lambda$ be an eigenvalue $($perhaps not isolated in $\spec(H))$ of $H$, a period $p$ Jacobi matrix on a tree.  Then the DOS measure gives weight to $\{\lambda\}$ of the form $j/p$ with $j$ a positive integer.
\end{corollary}

\begin{remark}
We are not aware of any example of a periodic Jacobi matrix on a tree with a non-isolated eigenvalue. It would be interesting to find such an example or prove this is impossible.
\end{remark}

That $j>0$ follows from the fact that the trace $T$ is faithful (Theorem \ref{TA.2}).  Aomoto has a corollary to his index theorem that can be regarded as the most important result in his paper \cite{AomotoPoint}; we state it as a theorem.

\begin{theorem} [\cite{AomotoPoint}] \lb{T8.4} Let $H$ be a period $p$ Jacobi matrix on a regular tree. Then $H$ has no eigenvalues so its spectrum is purely absolutely continuous.
\end{theorem}

\begin{proof} Suppose that $\lambda$ is an eigenvalue. Let $d$ be the degree of the underlying tree and let $r$ be the total number of edges between $X^{(1)}_\lambda$ and $X^{(-1)}_\lambda$ so
\begin{equation}\label{8.2}
  r \le dp\left(X^{(-1)}_\lambda\right)
\end{equation}

Let $d_1(x)$ be the number of edges with both ends in $X^{(1)}_\lambda$ for which one end is $x$.  Thus
\begin{equation}\label{8.3}
  r = \sum_{x\in X^{(1)}_\lambda}(d-d_1(x)) = dp(X^{(1)}_\lambda)-2q(X^{(1)}_\lambda)
\end{equation}
Therefore, since $d\ge 2$, we that
\begin{equation}\label{8.4}
  dp(X^{(-1)}_\lambda) \ge dp(X^{(1)}_\lambda) - 2q(X^{(1)}_\lambda) \ge d[p(X^{(-1)}_\lambda)-q(X^{(1)}_\lambda)]
\end{equation}
so the right side of \eqref{8.1} is non-positive.  Since $T$ is faithful, $\lambda$ is not an eigenvalue after all.
\end{proof}

\begin{example} \lb{E8.5} Aomoto applied \eqref{8.1} to the model in Example \ref{E7.2}.  One has that $p(X^{(1)}_\lambda)=r$, $q(X^{(1)}_\lambda)=0$ and $p(X^{(-1)}_\lambda)=g$ so he gets a result equivalent to the weight of $\{0\}$ in the DOS measure being $(r-g)/(r+g)$ (he first needs a separate argument that $0$ is an eigenvalue).  This is, of course, what we found with our explicit calculation of Green's functions.  His argument also shows directly there is no pole in $G_g$.
\end{example}

Aomoto raises a general question which we turn to in a moment but there are already interesting open questions about the $rg$-model:

\begin{problem} \lb{P8.6} Does every periodic Jacobi matrix associated with a Jacobi matrix on the graph of an $rg$-model with $r\ne g$ have a point eigenvalue.
\end{problem}

\begin{problem} \lb{P8.7} If the answer to Problem \ref{P8.6} is yes and $r>g$ does every eigenvector live only on the $r$ vertices?
\end{problem}

There is the most general question asked by Aomoto:

\begin{problem} \lb{P8.8} Fix a finite connected leafless graph $\calG$.  If the universal cover of one Jacobi matrix on $\calG$ has a point eigenvalue, is that true of every other Jacobi matrix on $\calG$.
\end{problem}

\section{Borg's Theorem} \lb{s9}

In this section, we want to discuss a number of conjectures that we feel are among the most intriguing open questions in this area.  We begin with a first guess about what might be the analog of Borg's Theorem (Fact \ref{F2.14}):

\begin{guess} [Wrong!] \lb{G9.1} Let $H$ be a Jacobi matrix of period $p$ on a tree.  Suppose $\spec(H)$ has no gaps.  Then all $a$'s are equal to each other and all $b$'s are equal to each other.
\end{guess}

\begin{example} [$ac$ model] \lb{E9.2} Let $\calG$ be the graph with one vertex and two self loops so $p=1$, $q=2$ and the universal cover is the regular tree of degree $4$.  Let $J$ be the matrix with $b=0$ but two values $a$ and $c$ on the two edges.  Then $H$ has period $1$ so by Sunada's Theorem (Theorem \ref{T5.1}), there is no gap in the spectrum.  This is a counterexample to the strong guess that appears in Initial Guess \ref{G9.1}!  It also provides additional insight connected to Fact \ref{F2.11}.  $H(a,c)$ has spectrum $[-E(a,c),E(a,c)]$.  If $a=c=1/\sqrt{12}$, then $E=1$ by \eqref{4.2} with $d=4$. Moreover, if $a=0, c=1/2$, we also have $E=1$ by \eqref{4.2} with $d=2$.  It follows that for any $a\in (0,1/2)$, there is a $c(a)$ also in $(0,1/2)$ with $E(a,c(a))=1$, so we get a one parameter family of isospectral period $1$ Jacobi matrices on $\calT_4$. But for $a=1/\sqrt{12}$, the DOS is the scaled Kesten McKay distribution, \eqref{4.2}, with $d=4$ while as $a\downarrow 0$, the DOS converges to the distribution with $d=2$.  We presume the distributions are all different for $0<a<1/\sqrt{12}$.  We conclude that the analog of Fact \ref{F2.11} is false.
\end{example}

\begin{fact} \lb{F9.3} Unlike the $2$-regular case, for general periodic Jacobi matrices on trees, the graph $\calG$ and the spectrum do not determine the DOS!
\end{fact}

While we've only established this when $p=1$ and a regular tree of even degree, it surely must be true in great generality.

\begin{example} [$ace$ model] \lb{E9.4} Looking at the model in Example \ref{E9.2} one might worry it was a more general indication that Guess \ref{G9.1} fails not only in the case of $\calT_{2d}$ by also for $\calT_{2d+1}$.  So it is natural to consider the graph $\calG$ with $p=2$, $q=3$ of $2$ vertices, connected by $3$ edges with $b_i=0$ but three distinct values, say, $a, c, e$ on the three edges.  The corresponding $H$ lives on $\calT_{3}$ and has three edges with distinct values in the Jacobi parameters coming out of each vertex.  Perhaps this model also has no gap (although its period seems to be $2$).  In fact for many non-zero values of the three parameters, there is a gap.  Consider first what happens if $\calG$ has one edge removed.  Then the tree is $1D$ with all $b=0$ and alternating values on the edges of period $2$.  By Borg's theorem, this $1D$ model has a gap.  Indeed, that problem has discriminant $\Delta(z)=\frac{z^2-a^2-c^2}{ac}$. The edges of the spectrum are given by $\Delta(x)=\pm 2$ or $x = \pm(a\pm c)$.  If $c<a$, the spectrum is thus $(-a-c,-(a-c))\cup(a-c,a+c)$ which has a gap of size $2(a-c)$.  If $e=0$, the operator on $\calT_3$ degenerates into a direct sum of the $1D$ operators.  The $2\times 2$ matrix with 0 on diagonal and $e$ off-diagonal has norm $e$. By a standard argument from spectral theory the gap persists if $2(a-c)>2e$.  Thus if $a>c+e$, there is a gap.  We presume there is also a gap unless they are all equal.
\end{example}

With this example in mind, we break our conjecture about the analog of Borg's theorem into three parts:

\begin{conjecture} \lb{C9.5} A periodic Jacobi matrix on a tree which is not of constant degree always has a gap in its spectrum.
\end{conjecture}

\begin{conjecture}  \lb{C9.6} A periodic Jacobi matrix on an odd degree regular tree, $\calT_{2j+1}$,  with no gap in its spectrum has constant $a$'s  and constant $b$'s.
\end{conjecture}

\begin{conjecture} \lb{C9.7} A periodic Jacobi matrix on an even degree regular tree, $\calT_{2j}$, with no gap in its spectrum is of period $1$. \end{conjecture}

The last conjecture says that any such Jacobi matrix has constant $b$'s and $j$ (possibly not distinct) values of $a$ so that each vertex has $2$ of each of these $j$ values on the edges attached to it.

With our definition of period, there are no period $1$ Jacobi matrices on $\calT_{2j+1}$ (or on non-regular trees).  This leads us to propose

\begin{problem} \lb{P9.8} Find a definition of period for periodic Jacobi matrices on trees so that gap labeling holds and so that a Jacobi matrix on $\calT_{2j+1}$ has period $1$ if and only if $a$ and $b$ (as functions on the edges and vertices respectively) are each constant and so that non-regular trees have no period $1$ Jacobi matrices.
\end{problem}

If that is done the union of the three conjectures above is that a gapless periodic Jacobi matrix on a tree has period $1$.  With a proper definition of period one can also hope to extend the Borg--Hochstadt theorem:

\begin{conjecture} \lb{C9.9} There is a definition of period for Jacobi matrices on trees so that gap labeling holds and so that the following analog of the Borg--Hochstadt Theorem holds: If the IDS of a periodic Jacobi matrix on a tree has a value $j/p$ in each gap of the spectrum, then the period is (a divisor of) $p$.
\end{conjecture}

\section{Additional Conjectures and Problems} \lb{s10}

In this final section, we discuss a number of conjectures and open problems in the spectral theory of periodic Jacobi matrices on trees.

\subsection{The Riemann surface of $G$} On the basis of the analytic structure in the $1$D case, there are some natural guesses about the more general case:

\begin{problem} \lb{C10.3} Are all the branch points of the $m-$ and Green's functions are square root so the varieties defined by them are manifolds (i.e. have no singularities).
\end{problem}

\begin{problem} \lb{C10.4} Are the $m-$ and Green's functions are two sheeted. \end{problem}

These are consistent with all the examples where we can do calculations although we haven't any other strong reasons in support of them.  At one point we conjectured that all branch points are on the real axis, but on the basis of some examples we hope to discuss in detail elsewhere, we no longer expect that is true.

In the $1$D case, one argues that $G_j$ has a single zero in each gap. Those zeros are associated to poles of either $m_+$ or $m_-$ and, then, the $m_-$ poles to second sheet poles of $m_+$.

\begin{problem} \lb{P10.4A} Explore what connection there is between non-physical sheet poles (i.e.\ poles that lie on the sheets associated to the analytic continuation, and not on the original domain of definition) of an $m_j^\beta$ and physical sheet poles of the other $m_k^\alpha$ in \eqref{6.5}.  Resolve the notion that there are $d-1$ such functions and, we suspect, only two branches for $m$.
\end{problem}

\subsection{Open gaps}
Let $\calG$ be a finite graph. Let $\calP(\calG)$ be the set of allowed Jacobi parameters. It is an open orthant of $\bbR^{p+q}$ since $p + q$ is the number of vertices plus the number of edges (it is only an orthant since all $a>0$).  We say a period $p$ Jacobi matrix has all gaps open if the spectrum has $p-1$ gaps.  It is easy to see the set of Jacobi parameters for which all gaps
are open is an open set in $\bbR^{p+q}$. We believe the most interesting open question except perhaps for Borg's Theorem is

\begin{conjecture} \lb{C10.5} The set of parameters with all gaps open is a dense open set in the set of allowed parameters. \end{conjecture}

We at least know the set is non-empty, for, if all b are different and $\sum_\alpha a_\alpha < \min_{i\ne j}  |b_i -  b_j|$, then all gaps are open.

\begin{conjecture} \lb{C10.6} The set of parameters where not all gaps are open is a variety of codimension $2$.
\end{conjecture}

The problem is we have no way of describing gap edges analogous to periodic and anti--periodic eigenvalues in the $1$D case.

\begin{problem} \lb{P10.7} Find an effective specification of gap edges.
\end{problem}

\subsection{Isospectral sets}
We've seen by example that unlike the $1$D case, two different periodic Jacobi matrices with the same tree, same period and same spectrum, can have different DOS (see Example \ref{E9.2}).

\begin{problem} \lb{P10.8} Classify the possible DOS allowed for a given tree, period and spectrum.
\end{problem}

The analog of having the same spectrum is the fine property of having the same DOS.  We then say they lie in the same IsoDOS set.

\begin{problem} \lb{P10.9} Is the IsoDOS set a manifold? Is it perhaps a torus?
\end{problem}

\begin{problem} \lb{P10.10} Is there an natural flow on the IsoDOS set?
\end{problem}

\subsection{Direct integral decomposition}
One of the most powerful tools in understanding the $1$D case is looking at the direct integral decomposition of the representation of the translation group into irreducibles which is simple because those irreducibles are one dimensional.  There has been considerable literature studying the representations of the free non-abelian groups and of the decomposition of translations on trees of which we mention Cartier \cite{Cartier}, Fig\`{a}-Talamanca-Steger \cite{FTS} and Woess \cite{Woess}.

\begin{problem} Determine if the direct integral decomposition is of any use in spectral analysis. In particular, do gap edges have anything to do with particular irreducible representations? \end{problem}

\appendix

\section{Proof of the Gap Labeling Theorem} \lb{appendix}

In this appendix, we both set notation and provide a proof of the fundamental gap labeling theorem accessible to spectral theorists.  The gap labeling theorem is close to a theorem of Pimsner--Voiculescu \cite{PV} and our proof follows that of Effros \cite{Eff}, which, in turn, simplifies an approach of Cunze \cite{Cun} and Connes \cite{Connes}.  We differ from Effros in discussing general $\ell$ and $\bbC^p$ valued functions, as well as providing some technical issues that he omits.  Sunada \cite{Sun} also provides an appendix with a proof but we feel our discussion here is more approachable, in part because we restrict to the situation relevant to the simpler lattice case discussed in this paper.

Because of our audience, we assume that the reader has knowledge of the basic facts about the trace class, $\calI_1$, in $\calL(\calH)$, the bounded operators on a separable Hilbert space, $\calH$, including the definition of the trace map, $\tr$.  These basics can be found, for example, in the books of Goh'berg--Krein \cite{GK} or Simon \cite{TI} or in Simon \cite[Chapter 3]{OT}.  We will need the following result

\begin{theorem} \lb{TA.1} Let $P, Q$ be orthogonal projections in $\calL(\calH)$ so that $P-Q\in\calI_1$.  Then $\tr(P-Q)\in\bbZ$.
\end{theorem}

This was first proven by Effros \cite{Eff} and rediscovered, with a different proof, by Avron--Seiler--Simon \cite{ASS}.  \cite{OT} has three proofs: the original proof of Effros \cite[Problem 3.15.20]{OT}, the proof of Avron et.\ al.\ \cite[Theorem 3.15.20]{OT} and an otherwise unpublished proof using the Krein spectral shift \cite[Problem 5.9.1]{OT}.  We note the simple proof of \cite{ASS}: Let $A=P-Q, B=\bdone-P-Q$.  Simple algebra shows that
\begin{equation}\label{A.1}
  A^2+B^2=\bdone, \qquad AB=-BA
\end{equation}

Since $A$ is assumed trace class and self-adjoint trace class operators have an orthonormal basis of eigenvectors \cite[Section 3.2]{OT}, if, for $\lambda \in\calE(A)$, the non--zero eigenvalues of $A$, we define $\calH_\lambda = \{\varphi\,|\,A\varphi=\lambda\varphi\}$,  we have that
\begin{equation}\label{A.2}
  \tr(P-Q) = \sum_{\lambda\in\calE(A)} \lambda \dim(\calH_\lambda)
\end{equation}
By \eqref{A.1}, $B$ maps $\calH_\lambda$ to $\calH_{-\lambda}$ and if $\varphi\in\calH_\lambda$, we have $B^2\varphi=(1-\lambda^2)\varphi$, so if $\lambda\ne\pm 1$ then $B$ is an invertible map of $\calH_\lambda$ to $\calH_{-\lambda}$ and thus
\begin{equation}\label{A.3}
  \lambda\ne\pm 1 \Rightarrow \dim(\calH_\lambda)=\dim(\calH_{-\lambda})
\end{equation}
Thus, by \eqref{A.2},
\begin{equation}\label{A.4}
  \tr(P-Q) = \dim(\calH_1)-\dim(\calH_{-1}) \in \bbZ
\end{equation}

Now, fix $\ell\in\{1,2,\dots\}$ and let $\bbF_\ell$ be the free group on $\ell$ generators (which is non-abelian if $\ell>1$ and $\bbZ$ if $\ell=1$).  We let $\calT_{2\ell}$ be the regular tree of degree $2\ell$.  One can associate the vertices in $\calT_{2\ell}$ with $\bbF_\ell$ (so that $\calT_{2\ell}$ becomes the Cayley graph of the group).  To do this, first pick an orientation so that each vertex in $\calT_{2\ell}$ has $\ell$ edges coming out of it and $\ell$ edges coming in; for example, start at one vertex and inductively define orientations starting with the chosen vertex.

If $\bbF_\ell$ is generated by $x_1,\dots,x_\ell$ and their inverses, we label the $\ell$ edges coming out of each vertex with $x_1,\dots,x_\ell$ and then view the $\ell$ coming in with $x_1^{-1},\dots,x_\ell^{-1}$ so that an edge coming out of one vertex as $x_j$ comes into its other vertex as $x_j^{-1}$.

Given this tree with labels, we can define $\ell$ maps of $\calT_{2\ell}$ to itself, $\tau(x_j),\,j=1,\dots,\ell$ by mapping a vertex $w$ into the vertex connected to it by the edge that starts at $w$ and is labeled $x_j$.  Any element, $w$ in $\bbF_\ell$, is uniquely associated to a word $y_1\dots y_k$ where each $y_m$ is an $x_j$ or $x_j^{-1}$ with the rule that no $x_j$ is next to an $x_j^{-1}$. We associate $x_j^{-1}$ to the map $\tau(x_j)^{-1} \equiv \tau(x_j^{-1})$ and then define $\tau(w) \equiv \tau(y_1)\dots\tau(y_k)$.  $\tau$ defines a free transitive action of the group $\bbF_\ell$ on the set $\calT_{2\ell}$.  By picking, once and for all, a vertex, $e_0$ in $\calT_{2\ell}$ to associate with $e$, the identity in $\bbF_\ell$, the map $w\mapsto \tau(w)e_0$ defines a bijection, $\sigma$, of $\bbF_\ell$ onto $\calT_{2\ell}$ so that
\begin{equation}\label{A.4A}
  \tau(w_1)\sigma(w_2) = \sigma(w_1 w_2)
\end{equation}

Fix $p\in\{1,2,\dots\}$. Our basic Hilbert space, $\calH_{2\ell;p}$ will be $\ell^2(\calT_{2\ell},\bbC^p)$ of square summable functions on $\calT_{2\ell}$ with values in $\bbC^p$. Associating to any $g \in \bbF_\ell$, the unitary map $U_0(g): f\in\calH_{2\ell;p}\mapsto U_0(g)f\in\calH_{2\ell;p}$ by
\begin{equation}\label{A.4B}
  (U_0(g)f)_{\sigma(w)} = f_{\sigma(g^{-1}w)}
  \end{equation}
defines a natural unitary representation of $\bbF_\ell$ which is just the direct sum of $p$ copies of the (left) regular representation.

We let $\calV_{2\ell;p}$ be the von Neumann algebra of all bounded operators, $B$, which commute with $\{U_0(g)\,|\,g\in\bbF_\ell\}$.  For any such $B$ there is a function $\widehat{B}: \bbF_\ell \rightarrow \hom(\bbC^p)$, the $p\times p$ matrices, so that
\begin{equation}\label{A.4C}
  (Bf)(g) = \sum_{h\in\bbF_\ell} \widehat{B}(h^{-1}g)\,f(h)
\end{equation}
Thus, for $\varphi\in\bbC^p$, $\widehat{B}(g)\varphi = B(\delta_{e;\varphi})(g)$, where $\delta_{e;\varphi}$ is the function that is supported at the identity and has the value $\varphi$ there.  It follows that $\widehat{B}$ is in $\ell^2(\calT_{2\ell},\hom(\bbC^p))$ so that the sum in \eqref{A.4C} converges.  If $\widehat{B}$ is supported at $g$, it acts like a right multiplication operator.

We define the unnormalized trace, $T:\calV_{2\ell;p}\to\bbC$ by
\begin{equation}\label{A.4D}
  T(B) = \tr(\widehat{B}(e))
\end{equation}
where $\tr$ is the trace on $p\times p$ matrices.

\begin{theorem} \lb{TA.2} $T$ is a positive, finite, faithful trace on $\calV_{2\ell;p}$ with $T(\bdone)=p$.
\end{theorem}

\begin{proof}  For $w\in\bbF_\ell$ and $\varphi\in\bbC^p$, let $\delta_{w;\varphi}$ be the function in $\calH_{2\ell;p}$ which is supported only at $w$ and has the value $\varphi$ there.  Then
\begin{equation}\label{A.5}
  \jap{\delta_{v;\psi},B\delta_{w;\varphi}}_{\calH_{2\ell;p}} = \jap{\varphi,\widehat{B}(w^{-1}v)\varphi}_{\bbC^p}
\end{equation}

Let $\left \{\varphi_j \right \}_{j=1}^p$ be the canonical basis for $\bbC^p$ and let $B_{v,w;j,k} \equiv \jap{\delta_{v;\varphi_j},B\delta_{w;\varphi_k}} = \widehat{B}(w^{-1}v)_{jk}$.  Then, for $B, C \in \calV_{2\ell;p}$, we have that
\begin{align}
  T(BC) &= \sum_{j=1}^{p} [\widehat{BC}(e)]_{jj} \nonumber  \\
        &= \sum_{j,k=1}^{p} \sum_{w\in\bbF_\ell} B_{e,w;jk}C_{w,e;kj} \nonumber \\
        &= \sum_{j,k=1}^{p} \sum_{w\in\bbF_\ell} \widehat{B}(w^{-1})_{jk}\widehat{C}(w)_{kj}   \label{A.5A} \\
        &= \sum_{j,k=1}^{p} \sum_{w\in\bbF_\ell} \widehat{B}(w)_{jk}\widehat{C}(w^{-1})_{kj}   \label{A.6} \\
        &= T(CB) \label{A.6A}
\end{align}
proving that $T$ is an (obviously finite) trace. \eqref{A.6} comes from the fact that summing over all $w$ is the same as summing over all $w^{-1}$ and \eqref{A.6A} comes from undoing all the steps that led to \eqref{A.5A}.

Next notice that using \eqref{A.5} twice, we have that
\begin{align*}
  \jap{\psi,\widehat{B^*}(w)\varphi} &= \jap{\delta_{w;\psi},B^*\delta_{e;\varphi}} \\
                                 &= \overline{\jap{\delta_{e;\varphi},B\delta_{w;\psi}}} \\
                                 &= \overline{\jap{\varphi,\widehat{B}(w^{-1})\psi}} \\
                                 &= \jap{\psi,\widehat{B}(w^{-1})^*\varphi}
\end{align*}
proving that
\begin{equation}\label{A.6B}
  \widehat{B^*}(w) = \widehat{B}(w^{-1})^*
\end{equation}
This and \eqref{A.5A} show that
\begin{equation}\label{A.6C}
  T(B^*B) = \sum_{w\in\bbF_\ell} \tr(\widehat{B}(w)^*\widehat{B}(w))
\end{equation}
which implies that $T$ is positive and that $T(B^*B)=0\Rightarrow \widehat{B} \equiv 0 \Rightarrow B=0$, so $T$ is faithful.
\end{proof}

We let $\calC^{(0)}_{2\ell;p}$ be the set of those $B\in\calV_{2\ell;p}$ with $\{w\,|\,\widehat{B}(w)\ne 0\}$ finite.  It is generated by the set of diagonal $B$ (i.e.\ $B$'s with $\widehat{B}(w) = 0$ if $w \ne e$) and $\left \{U_0(x_j),U_0(x_j^{-1});\,j=1,\dots,\ell \right \}$.  $C^*_{red}(\bbF_\ell;\bbC^p)$, the reduced $C^*$ algebra of $\bbF_\ell$ (when $p=1$), is the $C^*$ closure of $\calC^{(0)}_{2\ell;p}$.  One has that $C^*_{red}(\bbF_\ell;\bbC^p) = C^*_{red}(\bbF_\ell;\bbC)\otimes\hom(\bbC^p)$.  The main gap labeling theorem is

\begin{theorem} \lb{TA.3} Let $P\in C^*_{red}(\bbF_\ell;\bbC^p)$ be an orthogonal projection.  Then
\begin{equation}\label{A.7}
  T(P) \in \bbZ
\end{equation}
\end{theorem}

\begin{corollary} \lb{CA.4} If $P\in C^*_{red}(\bbF_\ell;\bbC)$ is an orthogonal projection, then $P=0$ or $P=1$.
\end{corollary}

\begin{remark}  This is the celebrated theorem of Pimsner--Voiculescu \cite{PV} which established a conjecture of Kadison.
\end{remark}

\begin{proof} [Proof (Of the Corollary)] Since $0\le P \le \bdone$, we have that $T(P)$ is $0$ or $1$, so either $T(P)=0$ or $T(\bdone-P)=0$.  Since $T$ is faithful, either $P=0$ or $\bdone-P=0$.
\end{proof}

To prove Theorem \ref{TA.3}, we introduce a degenerate representation, $U_1$, of $\bbF_\ell$.  Single out $e_0\in\calT_{2\ell}$ and remove it.  We get $2\ell$ rooted trees with roots corresponding to the generators and their inverses. Join the $x_j^{-1}$ tree to the $x_j$ tree  and get $\ell$ copies of the original tree which can be directed and labeled consistently with the original directions and labels.  We thus find that
\begin{equation}\label{A.8}
  \calH_{2\ell;p} = \bbC^p\oplus\,\,\ell\text{ copies of } \calH_{2\ell;p}
\end{equation}
$U_1$ is then $0\oplus\,\,\ell\text{ copies of } U_0$.  It is a degenerate representation since the identity in $\bbF_\ell$ goes into a codimension $1$ projection rather than the identity on $\calH_{2\ell;p}$.

Both $U_0$ and $U_1$ extend to representations $\pi_0$ and $\pi_1$ (with $\pi_1$ degenerate) of $\calV_{2\ell;p}$ and so, by restriction, of $C^*_{red}(\bbF_\ell;\bbC^p)$.  $\pi_0(B)=B$ and $\pi_1(B)$ is a bounded operator on $\calH_{2\ell;p}$ (which is not usually in $\calV_{2\ell;p}$).

\begin{proposition} \lb{PA.5} (a) the diagonal part of $\pi_1(B)$ at site $w$ is
\begin{equation}\label{A.9}
  \widehat{B}(e) \text{ if } w\ne e; \qquad 0 \text{ if } w=e
\end{equation}

(b) If $P\in C^*_{red}(\bbF_\ell;\bbC^p)$ is an orthogonal projection, so are $\pi_0(P)$ and $\pi_1(P)$.

(c) If $B\in C^*_{red}(\bbF_\ell;\bbC^p)$ is such that
\begin{equation}\label{A.10}
   \pi_0(B)-\pi_1(B)\in\calI_1(\calH_{2\ell;p})
\end{equation}
then
\begin{equation}\label{A.11}
  T(B) = \tr(\pi_0(B)-\pi_1(B))
\end{equation}

(d) If $P\in C^*_{red}(\bbF_\ell;\bbC^p)$ is an orthogonal projection so that
\begin{equation}\label{A.12}
  \pi_0(P)-\pi_1(P)\in\calI_1(\calH_{2\ell;p})
\end{equation}
then
\begin{equation}\label{A.13}
  T(P)\in \bbZ
\end{equation}
\end{proposition}

\begin{proof} (a) The diagonal part of $\pi_0$ is just $\widehat{B}(e)$.  Since $\pi_1$ is just a copy of $\pi_0$ on each of the $\ell$ trees, the first assertion in \eqref{A.9} follows.  The second assertion is immediate.

(b) follows from the fact that $\pi_0$ is a representation and $\pi_1$ a degenerate representation.

(c) If $C\in\calI_1$, its trace is just the sum over $w\in\bbF_\ell$ of $\tr_{\bbC^p}(C_{w w})$.  By (a)
\begin{equation*}
  [\pi_0(B)-\pi_1(B)]_{w w} = \jap{\delta_{w; e},\widehat{B}(e)\delta_{w; e}}
\end{equation*}
so \eqref{A.11} is immediate.

(d) follows from (c) and Theorem \ref{TA.1}
\end{proof}

Let
\begin{equation}\label{A.14}
  \calA_0 = \{B\in C^*_{red}(\bbF_\ell;\bbC^p)\,|\,\pi_0(B)-\pi_1(B)\in\calI_1(\calH_{2\ell;p})\}
\end{equation}
Since $T$ is operator norm continuous on $C^*_{red}(\bbF_\ell;\bbC^p)$, Theorem \ref{TA.3} follows from Proposition \ref{PA.5}(d) and

\begin{theorem} \lb{TA.6} Any orthogonal projection in $C^*_{red}(\bbF_\ell;\bbC^p)$ is an operator norm limit of orthogonal projections in $\calA_0$.
\end{theorem}

We begin the proof of Theorem \ref{TA.6} with

\begin{lemma} \lb{LA.7} Let $f$ be a continuous function from $[0,1]$ to $\calI_1$ (continuous in $\calI_1$-norm).  Then the Riemann sums for $\int_{0}^{1} f(s)\,ds$ converge in $\calI_1$-norm.
\end{lemma}

\begin{proof} Follows from the standard proof of convergence of the Riemann integral \cite[Theorem 4.1.1]{RA} and the fact that continuity on $[0,1]$ implies uniform continuity.
\end{proof}

Recall \cite[Section 2.3]{OT}, that in any Banach algebra, $\calB$, if $x\in\calB$ and $E$ is a closed subset of the spectrum, $\sigma(x)$, of $x$, then
\begin{equation}\label{A.15}
  p=(2\pi i)^{-1}\int_\Gamma \frac{dz}{z-x}
\end{equation}
(here $\Gamma$ is a curve that winds around $E$) defines a natural Gel'fand projection which when $x=B\in\calL(\calH)$ and $B$ is self-adjoint agrees with the spectral projection, $P_E(B)$.

\begin{proposition} \lb{PA.8} (a) $\calA_0$ is a *-algebra.

(b) $\calC^{(0)}_{2\ell;p} \subset \calA_0$

(c) If $B\in C^*_{red}(\bbF_\ell;\bbC^p)$ lies in $\calA_0$ and $z\ne 0$ with $z \notin\sigma(A)$, then $(B-z)^{-1} \in \calA_0$ and on $\bbC\setminus\sigma(B)$, $z\mapsto\pi_0((B-z)^{-1})-\pi_1((B-z)^{-1})$ is $\calI_1$-norm continuous.

(d) If $E$ is a closed subset of the spectrum of $B\in\calA_0$ with $0 \notin E$, then $P$ given by \eqref{A.15} lies in $\calA_0$.
\end{proposition}

\begin{proof} (a) That $B\in\calA_0\Rightarrow B^*\in\calA_0$ follows from the fact that $\calI_1$ has that property.  That $A,B\in\calA_0 \Rightarrow AB\in\calA_0$ follows from the fact that $\calI_1$ is an ideal in $\calL(\calH)$ and
\begin{equation*}
  \pi_0(AB)-\pi_1(AB) = \pi_0(A)(\pi_0(B)-\pi_1(B))+(\pi_0(A)-\pi_1(A))\pi_1(B)
\end{equation*}

(b) If $B$ is a generator of $\calC^{(0)}_{2\ell;p}$, $C=\pi_0(B)-\pi_1(B)$ is finite rank (because $C_{w,v;j,k} = 0$ unless both $w$ and $v$ are $e$ or one of its neighbors).  By (a), $\calC^{(0)}_{2\ell;p}\subset\calA_0$.

(c) We note that $(\pi_1(B)-z)^{-1} = \pi_1((B-z)^{-1})-z^{-1}P_{e_0}$ with $P_{e_0}$ the projection onto those functions supported at $e_0$ while $(\pi_0(B)-z)^{-1} = \pi_0((B-z)^{-1})$.  Thus
\begin{align}
  \pi_0(&(B-z)^{-1})-\pi_1((B-z)^{-1}) =  \nonumber \\
           &-\pi_0((B-z)^{-1})[\pi_0(B)-\pi_0(B)](\pi_1(B)-z)^{-1})-z^{-1}P_{e_0} \lb{A.16}
\end{align}
which proves both assertions.

(d) By \eqref{A.16}, the integrand in
\begin{equation*}
  \pi_0(P)-\pi_1(P) = -(2\pi i)^{-1} \int_\Gamma[\pi_0((B-z)^{-1})-\pi_1((B-z)^{-1})]\,dz
\end{equation*}
is continuous in $\calI_1$, so by Lemma \ref{LA.7}, $P\in\calA_0$.
\end{proof}

\begin{proof} [Proof of Theorem \ref{TA.6}] Let $P \in C^*_{red}(\bbF_\ell;\bbC^p)$ be an orthogonal projection.  Since $\calC^{(0)}_{2\ell;p}$ is dense in $C^*_{red}(\bbF_\ell;\bbC^p)$, we can find $B_n$, self-adjoint, in $\calC^{(0)}_{2\ell;p}$ so that $\norm{B_n-P} < 1/2n$.  Since $\norm{B_n-P} < 1/2$, there is a gap about $\tfrac{1}{2}$ in $\sigma(B_n)$, so the spectral projection $Q_{\left(\tfrac{1}{2},\tfrac{3}{2}\right)}(B_n)$ is given by an integral like \eqref{A.15}.  Thus  $Q_{\left(\tfrac{1}{2},\tfrac{3}{2}\right)}(B_n)\in\calA_0$.  Since $Q_{\left(\tfrac{1}{2},\tfrac{3}{2}\right)}(B_n)$ is given by this integral, as $n\to\infty$, we have that $Q_{\left(\tfrac{1}{2},\tfrac{3}{2}\right)}(B_n)\to Q_{\left(\tfrac{1}{2},\tfrac{3}{2}\right)}(P) = P$ in norm.
\end{proof}

This completes the proofs of the results.  We end this appendix with a remark that shows some of the technicalities of proof are needed.  Let $p=1$.  As we've seen, for any projection, $P$, in $\calV_{2\ell;p=1}$ $\pi_0(P)-\pi_1(P)$ has only one non-zero diagonal element in the natural basis of $\ell^2(\calT_{2p})$ and its value is $T(P)$.  Surely this must be some kind of trace of a difference of projections and so an integer.  But if $J_0$ is the free Jacobi matrix on the tree, it has continuous spectrum on $[-2\sqrt{2\ell},2\sqrt{2\ell}]$ (see \eqref{4.2}).  For any $\lambda\in (-2\sqrt{2\ell},2\sqrt{2\ell})$, the spectral projection $P_{(-\infty,\lambda)}(J_0)$ lies in $\calV_{2\ell;p=1}$ and has $T$ value, $k(\lambda)$ which is not an integer.  Thus Theorem \ref{TA.3} does not extend to $P\in\calV_{2\ell;p=1}$ and requires the $C^*_{red}(\bbF_\ell;\bbC^p)$ condition.



\section{Addendum} \lb{addendum}

Garza-Vargas and Kulkarni \cite{GVK} (henceforth GK) have noted that a result of Fig\'{a}-Talamanca-Steger \cite{FTS} (henceforth FTS) can be used to show that Conjectures \ref{C9.6}, \ref{C9.7} and \ref{C10.6} are all false! And GK also found a counterexample to Conjecture \ref{C9.5}. In this addendum, we want to explain this example in detail using the equations from our paper and propose possible replacement conjectures.  Before that, we want to mention that connected to our results on absence of singular continuous spectrum, there are earlier results on different but related models in Keller, Lenz and Warzel \cite{KLW2, KLW3}.  FTS \cite{FTS} studied the degree $d$ homogeneous tree with no potential (i.e. $b\equiv 0$) and all the same $d$ values on the vertices leaving each vertex.  Here's a key result (they call it a lemma):

\begin{theorem} \lb{TB.1} Consider a graph Jacobi matrix with $2$ vertices and $d$ edges between them and $a_1\ge a_2\ge\ldots\ge a_d$ on the edges (and all $b's=0$).  Let $H$ be the lifted Jacobi matrix on the universal cover (so a degree $d$ homogenous tree with potentially $d$ different $a$'s on the edges).  Then $0\in\spec(H)$ if and only if
\begin{equation}\label{B1}
  a_1^2 \le \sum_{j=2}^{d} a_j^2
\end{equation}
\end{theorem}

\begin{remarks} 1. Of course, if $d=2$ this only holds if $a_1=a_2$ but when $d\ge 3$, there are many examples of unequal $a$'s where \eqref{B1} holds, e.g. if $d-1$ a's all equal $1$, it holds if the remaining one lies between $0$ and $\sqrt{d-1}$.

2. As noted in GK \cite{GVK}, by gap labelling and period $2$, there is at most one gap.  Moreover, since $H$ is unitarily equivalent to $-H$, the gap is open if and only if $0\notin\spec(H)$.  Thus this provides many counterexamples to Conjectures \ref{C9.6} and \ref{C9.7}.  To understand why Conjecture \ref{C10.6} has counterexamples, we need to consider $b_1$, $b_2$ also.  The gap remains absent if $b_1=b_2$, so we get a $d+1$ dimensional set with no gap in a space of total dimension $d+2$, so codimension $1$, not $2$. By Example \ref{E7.4}, if all $a$'s are equal, it opens when $b_1\ne b_2$ but from what we can prove, it is possible, although we think not likely, that there is an open set of parameters in the full $d+2$-dimensional parameter space with no gap open.

3. Garza Vargas-Kulkarni \cite{GVK} also construct a rather different period $2$ model that provides a counterexample to the non-homogeneous tree conjecture, Conjecture \ref{C9.5}.  Their underlying graph has two vertices and three edges with $b_1=b_2=0$ and $a$'s given by \\

\centerline{\includegraphics[scale=.25]{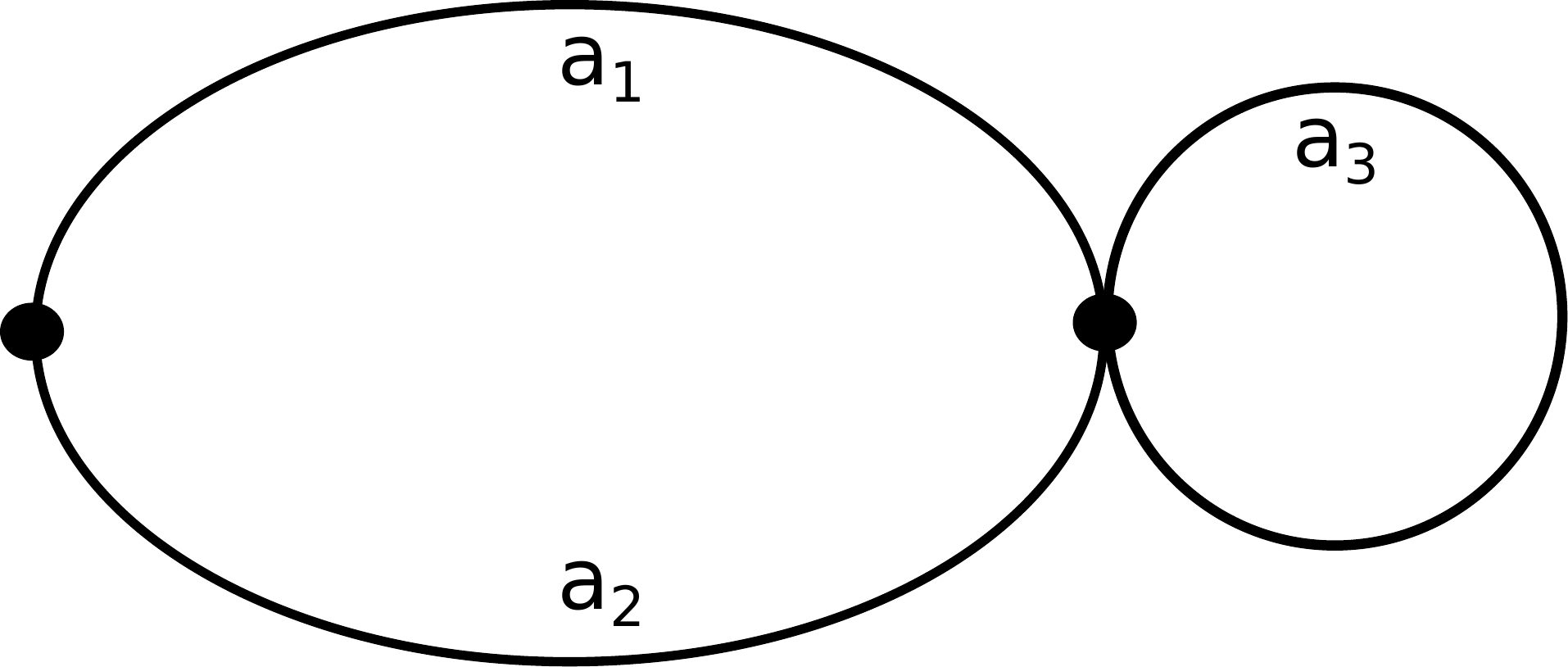}}
\medskip
If $a_1=a_2$, one can find a $1$-dimensional free model inside it but where even vertices have trees growing out of them. It is easy to write down a sequence of unit vectors, $\varphi_n$, with $\norm{H\varphi_n}\to 0$ (see \cite{GVK}) proving that $0\in\spec(H)$, which implies that there is no gap open.
\end{remarks}

The notation is made simpler if one realizes, as noted in \cite{FTS} that the degree $d$ homogeneous tree is a Cayley graph of the free product of $d$ copies of $\bbZ_2$.  Namely, take the group with $d$ generators, $x_1,\ldots,x_d$ with the only relations being that $x_j^2=e$ for all $j$.  Thus the group is all words $y=x_{j_1}x_{j_2}\ldots x_{j_\rho}$ with no two successive $j$ indices the same.  One has the obvious product and inverse.  One can define $\rho(y)=\rho$ which is the distance in the tree.  Thus the vertices of the tree can be associated to the points in a group, $\calG$, and our Hilbert space with $\ell^2(\calG)$.  There is an edge between $y_1, y_2 \in G$ if and only if $y_1=y_2x_j$ for some $j$. Left multiplication acts on the tree in such a way  that it also takes edges into edges.  This gives us a set of symmetries of the tree that acts freely and transitively on the vertices (but the action is not transitive on the edges).

$H$ is a matrix acting on $\l^2(\calG)$ with
\begin{equation}\label{B2}
  H_{y,yx{_j}} = a_j
\end{equation}
and all other matrix elements zero.  If $f\in\ell^2$ and we define the left regular representation of $\calG$ by
\begin{equation}\label{B3}
  (U(y)f)_w=f_{y^{-1}w}
\end{equation}
then $U$ commutes with $H$ and $U$ defines a free transitive action on the vertices of $\calT$, even for odd $d$.  Thus there is a sense in which this is period $1$, but gap labelling does \emph{not} hold for this definition of period.

We need some formulae for $G$ and $m$ from our paper as well as a formula for the off-diagonal Green's function that we didn't note explicitly in Section \ref{s6} but follow from the formulae there.  On general graphs, the $m$ functions depend on an edge and a particular vertex incident on the edge but, by the symmetry of this model, they only depend on the edge.  Thus we have one diagonal Green's function $G(z)$ and $d$ $m$-functions $m_j; \, j=1,\ldots,d$ \eqref{6.4} and \eqref{6.5} become
\begin{equation}\label{B4}
  G(z) = \frac{1}{-z-\sum_{j=1}^{d} a_j^2 m_j(z)}
\end{equation}
\begin{equation}\label{B5}
  m_j(z) = \frac{1}{-z-\sum_{k\ne j} a_k^2 m_k(z)}
\end{equation}

If $G_{w,y}(z)$ is the matrix element of the resolvent, then one has the following: if $y\in\calG$ so that $\rho(yx_j) = \rho(y)+1$ (i.e. the word formulae for $y$ doesn't end in $x_j$) then for $z\in\bbC_+$,
\begin{equation}\label{B6}
  G_{e,yx_j}(z) = -a_j m_j(z) G_{e,y}(z)
\end{equation}
(the analogs of this are well-known in the one-dimensional case). This follows from the off-diagonal term of the second equation after \eqref{6.2} above (if one writes the full $\ell^2$ space as a direct sum breaking at the edge that goes from $y$ to $yx_j$).  We will need special subsets, $S_j, j=1,\ldots,d$, of the group $\calG$ of all words of odd length the form $x_jx_{k_1}x_j\ldots x_jx_{k_m}x_j,\,k_q\ne j$ (a word of length $2m+1)$.

\begin{proposition} \lb{PB.2} Suppose that $0\notin\spec(H)$.  Then:
\begin{SL}
  \item[\rm{(a)}] $G(0)=0$, indeed, $G_{ey}(0)=0$ whenever $\rho(y)$ is even.
  \item[\rm{(b)}] All the $m_j$ functions have meromorphic continuations across a real neighborhood of $z=0$. There is a single $j$ among $1,2,\dots,d$ so that $m_j$ has a pole at z=0.  All the other $m_k$'s obey $m_k(0)=0$.
  \item[\rm{(c)}] $m_j(z)G(z)$ has a removable singularity at $z=0$ with value $-1/a_j^2$.  For $k\ne j$, $m_j(z)m_k(z)$ has a removable singularity at $z=0$ with value $-1/a_j^2$.
  \item[\rm{(d)}] We have the formula
  \begin{equation}\label{B7}
    G_{ey}(0)= \left\{
                 \begin{array}{ll}
                   \frac{1}{a_j}\prod_{i}^{m}-\left(\frac{a_{k_i}}{a_j}\right), & \hbox{ if } y= x_jx_{k_1}x_j\ldots x_jx_{k_m}x_j\in S_j\\
                   0, & \hbox{ if }y\notin S_j
                 \end{array}
               \right.
  \end{equation}
\end{SL}
\end{proposition}

\begin{proof} (a) If $W$ is the unitary on $\ell^2$ with $W\delta_y=(-1)^{\rho(y)}\delta_y$, then $WHW^{-1}=-H$ which implies that for $z\notin\spec(H)$, one has that $G_{ey}(z)=-(-1)^{\rho(y)}G_{ey}(-z)$.  Evaluating at $z=0$ yields the result.

(b) Since $-G^{-1}$ is Herglotz, and $G(0)=0$, we see that $-G^{-1}(z)$ is meromorphic across a real neighborhood of $z=0$ with a pole which has a negative residue there (since $-G^{-1}$ is real on the real axis near but away from zero, the residue is real and is given by $\lim_{\varepsilon\downarrow 0} (i\epsilon)i \Im(-G^{-1}(i\varepsilon))$ which is negative).  It follows from \eqref{B4} that, except perhaps at $x=0$, the $m$'s have imaginary parts going to zero so they are meromorphic and that at least one $m_k$ has a pole with negative residue there, suppose at least $m_j$.  \eqref{B5} then implies that the other $m_k$ must vanish at $z=0$.

(c) Because the functions are Herglotz, all the poles and zeros are simple.  It follows the singularities are removable.  By \eqref{B4} and \eqref{B5}, we can compute the limits as $z\to 0$.

(d) The formulae for $G_{ey}$ when $\rho(y)$ is even follow from (a) and when $\rho(y)=1$ by using (b) taking limits from \eqref{B6}.  By (b), for $k\ne j$, the limit from $\bbC_+$ as $z\to 0$ of $(-a_k m_k(z))(-a_j m_j(z))$ is $-a_k/a_j$.  The formulae for $G_{ey}(z)$ then follow by induction in $\rho(y)$
\end{proof}

Finally, to prove Theorem \ref{TB.1}, one needs the following result Haagerup \cite{Hag}:

\begin{theorem} \lb{BT.3} For $f, g\in\ell^2$, define $f*g(x)=\sum_{y\in G} f(xy^{-1})g(y)$.  Then
\begin{equation}\label{B8}
  \norm{f*g}_2 \le \norm{g}_2 \sum_{r=0}^{\infty} |n+1|\left(\sum_{\rho(y)=r}|f(y)|^2\right)^{1/2}
\end{equation}
\end{theorem}

\begin{proof} [Proof of Theorem \ref{TB.1}] We will prove
\begin{equation}\label{B9}
  0\notin\spec(H)\iff a_1^2 > \sum_{j=2}^{d} a_j^2
\end{equation}

Suppose first that $0\notin\spec(H)$.  By \eqref{B7},
\begin{equation}\label{B10}
    \sum_{y\,\mid\,\rho(y)=2M+1}|G_{ey}(0)|^2 =a_j^{-2} \left[a_j^{-2}\sum_{k\ne j}a_k^2\right]^M
\end{equation}
Since $0\notin\spec(H)\Rightarrow \sum_y |G_{ey}(0)|^2 < \infty$, we see that $a_j^{-2}\sum_{k\ne j}a_k^2<1\Rightarrow j=1$ and $a_1^2 > \sum_{j=2}^{d} a_j^2$.

Conversely suppose that $a_1^2 > \sum_{j=2}^{d} a_j^2$.  Let $f_j$ be the function on the right side of \eqref{B7}.  An easy calculation shows that
\begin{equation}\label{B11}
  \sum_{k=1}^{d} a_k f_j(yx_k) = \delta_{y e}
\end{equation}
that is $f_j$ is a formal Green's function of $H$ with eigenvalue 0. To see this, note that if $yx_j\in S_j$ with $y\ne e$, the sum has two terms which cancel and if $yx_j\notin S_j$, all terms in the sum  are zero.  Finally, if $y=e$, there is only one term in the sum which is $a_j(1/a_j)=1$.  If $a_1^2 > \sum_{j=2}^{d} a_j^2$, then by \eqref{B10}, convolution with $f_1$ is bounded on $\ell^2$ by \eqref{B8}.  By \eqref{B11}, $H(f_1*h)=h$ for any $h$ of finite support.  It follows that $0\notin\spec(H)$.
\end{proof}

We can replace the various conjectures that have counterexamples with alternatives.  So far, the only counter examples to Borg's theorem involve cases where $b$ is constant so the following might be true.

\begin{conjecture} \lb{BC.4} Consider a homogeneous degree $d$ periodic Jacobi matrix with $a\equiv 1$ and all gaps closed.  Then $b$=constant.
\end{conjecture}

We note that Borg's original theorem concerned $-d^2/dx^2+V(x)$, so one can argue this is closer to his original result than Conjecture \ref{C9.6}.  More generally one might guess that

\begin{conjecture} \lb{BC.5} If a periodic tree Jacobi matrix has all gaps closed then $b$ is constant.
\end{conjecture}

We make these conjectures with some caution bearing in mind that even a wrong conjecture can be useful \cite{Jito}.  As for Conjecture \ref{C10.6}, we note that if one looks at $1D$ periodic Schr\"{o}dinger operators (with $a=1$), those with not all gaps open have codimension $1$ when $p=2$ because there is only one non-trivial parameter (namely $b_2-b_1$) but once $p\ge 3$, the codimension should be $2$.  Perhaps the same is true on trees, that is the only case where Conjecture \ref{C10.6} is false is the case we know it is, namely when $p=2$!  So

\begin{problem} \lb{BP.6} Assuming conjecture \ref{C10.5} is correct and $p\ge 3$, is the set where not all gaps are open of codimension $1$ or codimension $2$?
\end{problem}

\makeatletter
\renewcommand\@bibitem[1]{\item\if@filesw \immediate\write\@auxout
    {\string\bibcite{#1}{B\the\value{\@listctr}}}\fi\ignorespaces}
\def\@biblabel#1{[B#1]}
\makeatother

\renewcommand\refname{Additional References}

\end{document}